\documentclass[10pt]{amsart}
\usepackage{amsmath,amsfonts,amssymb,color,graphics}
\usepackage[english]{babel}
\usepackage{enumerate, mathrsfs, xcolor, verbatim, tikz}
\usetikzlibrary{positioning, arrows.meta, calc}
\usepackage{geometry}
\usepackage[T1]{fontenc}
\usepackage[colorlinks=true,linkcolor=blue,citecolor=red,urlcolor=blue]{hyperref}

\tikzset{every node/.style={font=\normalsize}}

\def\build#1_#2^#3{\mathrel{\mathop{\kern 0pt#1}\limits_{#2}^{#3}}}

\DeclareMathOperator{\sat}{sat}

\newcommand{\Ac}{\mathcal{A}}

\numberwithin{equation}{section}

\def\iff{~if and only if~}

\newtheorem{theorem}{Theorem}[section]
\newtheorem{proposition}[theorem]{Proposition}
\newtheorem{lemma}[theorem]{Lemma}
\newtheorem{remark}[theorem]{Remark}
\newtheorem{example}[theorem]{Example}
\newtheorem{definition}[theorem]{Definition}
\newtheorem{assumption}[theorem]{Assumption}

\title[Rota-Baxter Operators in Discrete Control]{Study of Rota-Baxter Operators in Matrix $C^*$-Algebras Motivated by Toeplitz Structures, and Applications to Sliding Mode Control}
\author{Marwa Ennaceur}
\address{Department of Mathematics, College of Science,
         University of Ha'il, Ha'il 81451, Saudi Arabia}
\email{\tt mar.ennaceur@uoh.edu.sa}
\keywords{Rota-Baxter operators; matrix $C^*$-algebra; sliding mode control; 
discrete-time delay systems; Lyapunov stability; linear matrix inequality; 
Lie algebra; Toeplitz algebra; $\mathcal{L}_2$-gain; 
quasi-sliding mode; Bilinear Matrix Inequality}
\subjclass[2020]{16W99, 17B99, 46L05, 93D05, 93C55}
\begin{document}

\begin{abstract}
This paper studies Rota-Baxter operators on the matrix $C^*$-algebra $M_n(\mathbb{C})$,
motivated by the discrete Toeplitz algebra (whose role is purely heuristic; see
Remark~\ref{rem:toeplitz_scope}). We provide a structural classification of such
operators compatible with the $C^*$-norm, analyze their induced Lie brackets, and
apply them to deform system matrices in discrete-time delayed systems under sliding
mode control. Lyapunov-based Bilinear Matrix Inequality conditions, together with a
tractable linear reformulation via $Q=X^{-1}$, guarantee asymptotic stability on the
sliding manifold and $\mathcal{L}_2$-gain stability. The effective gain from
uncertainty $\delta$ to state $x$ is $\gamma/\sqrt{\mu}$ with
$\mu=\lambda_{\min}(-\mathcal{M})>0$ determined \emph{a posteriori}; minimizing
$\gamma$ alone does not minimize this bound, which holds under zero extended initial
conditions ($V_0=0$). We work under the standing assumption $m=n$ (square actuation);
a supplementary non-degenerate example with $m=1$, $n=2$ illustrates LMI feasibility
with $\Pi\neq0$. All algebraic results are proved directly in $M_n(\mathbb{C})$; no
infinite-dimensional reduction is used.
\end{abstract}
\maketitle
\tableofcontents

\section{Introduction}\label{sec:intro}
The discrete Heisenberg algebra arises as an associative algebra generated by elements
\(u, u^*, n\) subject to the relations
\[
[n, u] = -u, \quad [n, u^*] = u^*, \quad [u, u^*] = p,
\]
where \(p\) is the rank-one projection onto the vacuum state. Equipped with the
commutator \([x,y]=xy-yx\), this algebra inherits a Lie structure. Graded extensions
yield Lie superalgebras, broadening the theoretical scope \cite{Guo2012, Bai2019}.

A faithful representation on the Hilbert space \(\ell^2(\mathbb{N})\) identifies \(u\)
with the unilateral shift \(U\), \(u^*\) with its adjoint \(U^*\), and \(n\) with the
number operator \(N e_k = k e_k\). The norm closure of the generated algebra yields the
discrete Toeplitz \(C^*\)-algebra \cite{Douglas1998}. While this infinite-dimensional
setting provides a rich topological framework for spectral analysis and noncommutative
deformations \cite{Connes1994, Arveson1976}, practical control synthesis requires
finite-dimensional approximations. Specifically, we work with compressions of Toeplitz
operators to the subspace $\mathrm{span}\{e_0,\ldots,e_{n-1}\} \subset \ell^2(\mathbb{N})$;
the finite-dimensional algebra $M_n(\mathbb{C})$ serves as the concrete representation
algebra.

\begin{remark}[Scope of the Toeplitz motivation]\label{rem:toeplitz_scope}
The discrete Toeplitz $C^*$-algebra motivates the structural setting of this paper
but does not enter the proofs in an essential way, and its role is purely heuristic.
The Rota-Baxter identity \eqref{eq:RB} is imposed directly on linear maps
$P: M_n(\mathbb{C}) \to M_n(\mathbb{C})$ and reduces to an explicit system of quadratic
constraints on the matrix entries of $P$, verifiable by direct computation.

A rigorous reduction from the infinite-dimensional Toeplitz algebra to $M_n(\mathbb{C})$
would require the invariance of the subspace $\mathrm{span}\{e_0,\ldots,e_{n-1}\}$
under all operators in the range of $P$. This condition fails already for the most
natural generator: the unilateral shift $U$ maps $e_{n-1}$ to $e_n$, which lies outside
$\mathrm{span}\{e_0,\ldots,e_{n-1}\}$, so the subspace is not invariant under $U$.
Consequently, no Rota-Baxter identity from the Toeplitz algebra compresses to one in
$M_n(\mathbb{C})$ via this subspace, and any such reduction would additionally require
spectral tools (Wiener--Hopf factorization, Fredholm index) that are not developed here.
This extension is deferred to future work. All algebraic results in
Sections~\ref{sec:classification}--\ref{sec:spectral} are proved entirely within
$M_n(\mathbb{C})$; the Toeplitz algebra is cited only as the source of structural
intuition.
\end{remark}

Within this setting, Rota-Baxter operators play a central role. A linear map
\(P:\Ac\to\Ac\) satisfies the Rota-Baxter identity of weight \(\lambda\in\mathbb{R}\):
\begin{equation}\label{eq:RB}
P(x)P(y) = P(x P(y)) + P(P(x) y) + \lambda P(x y), \quad \forall\, x,y\in\Ac.
\end{equation}
Originally introduced to formalize integration and summation \cite{Guo2009},
Rota-Baxter operators also generate solutions to classical and modified Yang-Baxter
equations, thereby linking associative and Lie structures \cite{EbrahimiFard2008}. In
operator algebras, they induce algebraic deformations that can model cumulative effects,
filtering, or memory-dependent interactions.

In control theory, sliding mode control (SMC) is renowned for its robustness against
matched uncertainties and time delays \cite{Utkin1992, Edwards1998}. We consider
discrete-time delayed systems of the form
\begin{equation}\label{eq:system}
\begin{cases}
  x(k+1) = A x(k) + A_d x(k - \tau) + B u(k) + D \delta(k), \\
  s(k) = C x(k), \\
  u(k) = -K x(k) - \rho \, \sat\bigl(s(k)/\phi\bigr),
\end{cases}
\end{equation}
where \(x(k)\in\mathbb{R}^n\), \(u(k)\in\mathbb{R}^m\), $B \in \mathbb{R}^{n \times m}$,
$D \in \mathbb{R}^{n \times p}$, \(\tau\in\mathbb{N}\) is a known delay,
\(\delta(k)\in\mathbb{R}^p\) is a bounded uncertainty satisfying
\(\|\delta(k)\|_2 \leq \delta_{\max}\) for all $k$, and \(\sat(\cdot)\) denotes the
componentwise saturation function
\[
  \sat(v/\phi)_i = \begin{cases} v_i/\phi & \text{if } |v_i| \leq \phi, \\
  \mathrm{sgn}(v_i) & \text{if } |v_i| > \phi, \end{cases}
\]
with boundary layer width \(\phi>0\) to mitigate chattering in discrete time.

We assume throughout that \(m = n\) (the square case) and \(p = m\), i.e.,
$C \in \mathbb{R}^{m \times n}$ is a square matrix. This ensures that
$CB_P \in M_m(\mathbb{C})$ is a square matrix, and its invertibility
(Assumption~\ref{ass:invert}) is well-defined in the ordinary sense without recourse to
pseudoinverses. The hypothesis $p = m$ is adopted for notational convenience: it ensures
that the congruence transformation $\Sigma = \mathrm{diag}(Q, Q, I_p, I_n)$ in
Remark~\ref{rem:LMI_linear} produces a well-structured LMI with consistent block
dimensions. Physically, $p = m$ means the number of disturbance channels equals the
number of control inputs, a structural simplification that can be relaxed at the cost of
a more complex block structure in \eqref{eq:LMI_linear}. The case \(m < n\) (rectangular
\(B\)) or \(p < m\) (fat output map \(C\)) requires a separate treatment involving either
right-inverses or a reformulation of the equivalent control, and is deferred to future
work. The supplementary non-degenerate numerical example of Section~\ref{sec:nondeg}
relaxes this assumption to $m=1$, $n=2$ for purely illustrative purposes.

\begin{assumption}[Bounded initial conditions]\label{ass:bounded_init}
The initial state and initial delay segment are bounded:
$\sup_{k\geq 0}\|x(k)\|\leq r_0 < \infty$ and
$\sup_{k\geq 0}\|x(k-\tau)\|\leq r_{d,0}<\infty$ prior to closed-loop activation.
This holds whenever the open-loop system is BIBO stable and the initial data
$x(0)$ and $\{x(-\tau),\ldots,x(-1)\}$ are finite. In practice, $r_0$ and $r_{d,0}$
are set conservatively from the known initial conditions; they are treated as design
parameters in the reaching analysis of Section~\ref{sec:smc}.

\textbf{Important note on the design order.} The bound $r_0$ is fixed from open-loop
data before the control gains $K$ and $\rho$ are selected. However, the validity of
$r_0$ as a uniform bound throughout the closed-loop reaching phase is a condition that
is \emph{verified a posteriori} after $K$ is chosen, via condition
\eqref{eq:r0_sufficient}. There is therefore no strict \emph{a priori} non-circularity:
one must first fix $r_0$ and $\rho_{\max}$ from open-loop data, then choose $K$
satisfying \eqref{eq:K_strong}, and only then verify \eqref{eq:r0_sufficient}. If the
condition fails, $r_0$ must be increased (conservatively) or $K$ must be redesigned.
The full sequential procedure is spelled out in Remark~\ref{rem:design}.

When the open-loop system is not BIBO stable, this assumption must be imposed explicitly
as a constraint on the initial data.
\end{assumption}

All operator norms appearing in the SMC analysis
(Sections~\ref{sec:smc}--\ref{sec:stability}) are spectral (operator) norms
$\|\cdot\|_{\mathrm{op}} = \sigma_{\max}(\cdot)$, unless otherwise stated. The Frobenius
norm $\|\cdot\|_F$ appears only in the structural analysis of $P_+$ in
Section~\ref{sec:classification} (Example~\ref{ex:RB_Mn}(b) and
Remark~\ref{rem:friedrichs}), where it is always labeled explicitly.

\textbf{A non-standard norm requirement.} Throughout the SMC analysis, we impose
condition \eqref{eq:K_strong}: $\|A_P - B_PK\|_{\mathrm{op}} < 1$. This is \emph{strictly
stronger} than Schur stability (spectral radius $\rho(A_P-B_PK) < 1$), and is not
automatically satisfied even when $A_P - B_PK$ is Schur stable. Its necessity is
discussed in Remark~\ref{rem:alpha0_validity}; operator-specific guidance is given in
Remarks~\ref{rem:K_strong_scalar} and~\ref{rem:K_strong_Pplus}.

While the infinite-dimensional Toeplitz algebra inspires the structural properties of
memory operators (see Remark~\ref{rem:toeplitz_scope}), the control application is
carried out within the finite-dimensional matrix $C^*$-algebra \(\Ac = M_n(\mathbb{C})\).
By applying a Rota-Baxter operator \(P:\Ac\to\Ac\), we obtain deformed dynamics
\(A_P=P(A)\), \(A_{d,P}=P(A_d)\), \(B_P=P(B)\). This deformation preserves the
algebraic framework while introducing structured weighting. The weight \(\lambda\) governs
the algebraic nature of the deformation: \(\lambda=0\) corresponds to integration-like
behavior (splitting of the algebra), while \(\lambda\neq 0\) models persistent algebraic
feedback.

Our contributions are threefold:
\begin{enumerate}
  \item A structural characterization of Rota-Baxter operators on \(M_n(\mathbb{C})\),
        with explicit finite-dimensional examples and compatibility conditions for induced
        Lie brackets, together with a direct justification of the finite-dimensional setting.
  \item A rigorous Lyapunov-based stability analysis for delayed sliding mode systems under
        Rota-Baxter deformation, yielding a BMI condition via the Schur complement lemma,
        along with a fully linear reformulation via the congruence transformation $Q = X^{-1}$.
  \item An algebraic characterization of sliding surface invariance and spectral properties
        of the reduced dynamics, under explicit non-degeneracy and commutativity assumptions
        on \(P\), whose validity is discussed for the concrete examples provided.
\end{enumerate}

The remainder of the paper is organized as follows. Section~\ref{sec:preliminaries}
introduces the necessary algebraic and topological preliminaries.
Section~\ref{sec:classification} establishes structural properties of Rota-Baxter
operators, including explicit finite-dimensional examples. Section~\ref{sec:smc}
formulates the sliding mode control problem with Rota-Baxter deformation.
Section~\ref{sec:stability} presents the main stability theorem and matrix inequality
conditions. Section~\ref{sec:spectral} analyzes spectral and algebraic properties.
Section~\ref{sec:numerical} provides a numerical illustration, including a non-degenerate
example with $m=1$, $n=2$. Concluding remarks are given in Section~\ref{sec:conclusion}.

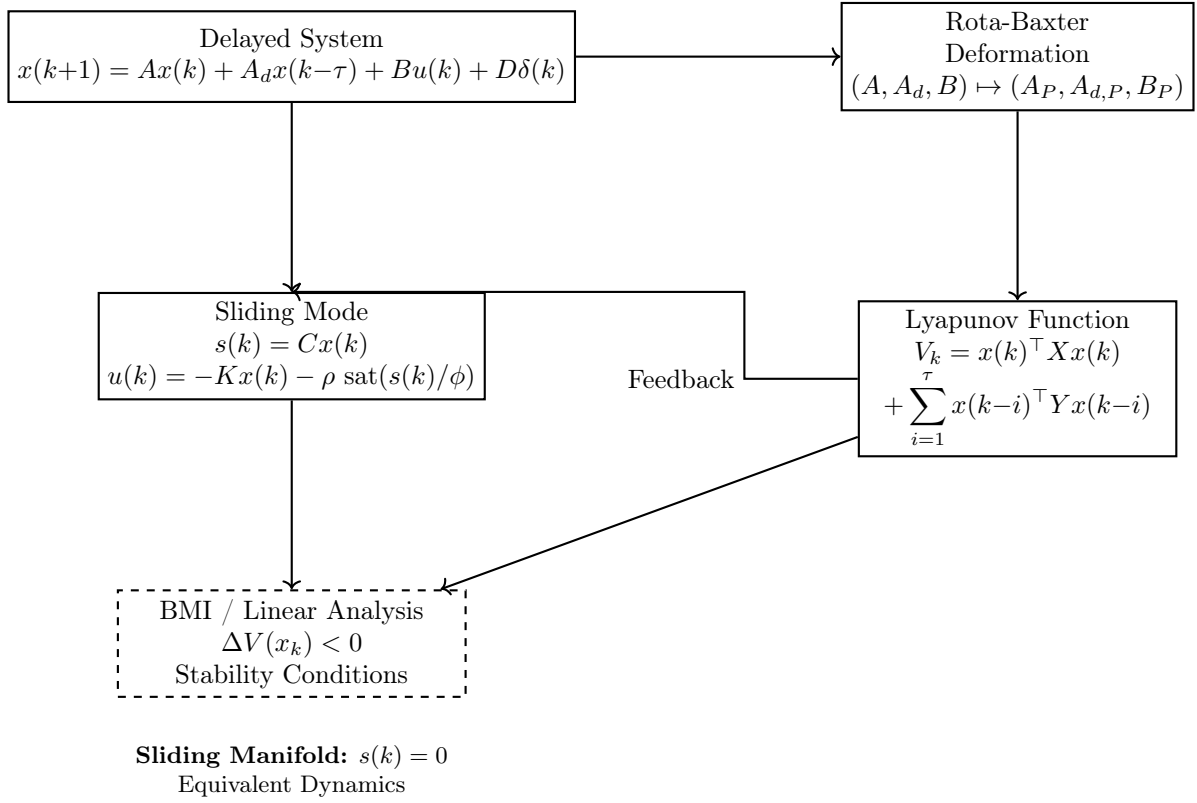
\begin{figure}[h!]
\centering
\begin{tikzpicture}[
  block/.style = {draw, thick, minimum height=1.2cm, minimum width=3.2cm, align=center},
  arrow/.style = {thick, ->},
  dashedblock/.style = {draw, dashed, thick, minimum height=1.4cm, minimum width=4.6cm, align=center}
]
\node[block] (system) {Delayed System\\$x(k{+}1) = A x(k) + A_d x(k{-}\tau) + B u(k) + D \delta(k)$};
\node[block, right=3.5cm of system] (RB) {Rota-Baxter\\Deformation\\$(A, A_d, B) \mapsto (A_P, A_{d,P}, B_P)$};
\node[block, below=2.5cm of system] (sliding) {Sliding Mode\\$s(k) = C x(k)$\\$u(k) = -K x(k) - \rho\, \sat(s(k)/\phi)$};
\node[block, below=2.5cm of RB, minimum width=4.2cm] (lyap) {Lyapunov Function\\$V_k = x(k)^\top X x(k)$\\$+\displaystyle\sum_{i=1}^\tau x(k{-}i)^\top Y x(k{-}i)$};
\node[dashedblock, below=2.5cm of sliding] (LMI) {BMI / Linear Analysis\\$\Delta V(x_k) < 0$\\Stability Conditions};

\draw[arrow] (system) -- (RB);
\draw[arrow] (system) -- (sliding);
\draw[arrow] (RB) -- (lyap);
\draw[arrow] (sliding) -- (LMI);
\draw[arrow] (lyap) -- (LMI);
\draw[arrow] (lyap.west) -- ++(-1.5,0) node[left]{Feedback} |- (sliding.north);

\node[below=0.5cm of LMI, align=center, font=\small] (manifold) {\textbf{Sliding Manifold:} $s(k)=0$ \\ Equivalent Dynamics};
\end{tikzpicture}
\caption{Conceptual framework: Rota-Baxter operators in sliding mode control of delayed
systems under uncertainty.}
\label{fig:framework}
\end{figure}

\section{Preliminaries}\label{sec:preliminaries}
This section introduces the fundamental algebraic and topological structures underlying
our study.

\subsection{Associative algebras and Lie structures}
An associative algebra \(\Ac\) over a field \(\mathbb{K}\) is a vector space equipped
with a bilinear product \(\cdot : \Ac \times \Ac \to \Ac\) satisfying
$(x \cdot y) \cdot z = x \cdot (y \cdot z)$ for all \(x,y,z \in \Ac\).

\begin{definition}
Given an associative algebra \(\Ac\), the commutator
\[
[x, y] := x y - y x
\]
endows \(\Ac\) with a Lie algebra structure \(\mathfrak{g}\), satisfying bilinearity,
antisymmetry, and the Jacobi identity \cite{Humphreys1972}.
\end{definition}

\begin{remark}
Graded extensions of this framework yield Lie superalgebras
\(\mathfrak{g} = \mathfrak{g}_{\bar{0}} \oplus \mathfrak{g}_{\bar{1}}\). While not
explicitly exploited in the finite-dimensional control synthesis below, this structure
provides a natural setting for future work on supersymmetric control and fermionic memory
models \cite{Kac1977}.
\end{remark}

\subsection{Topological and \(C^*\)-algebraic structures}
\begin{definition}
A \(C^*\)-algebra is a complex Banach algebra \((\Ac, \|\cdot\|)\) with an involution
\(*\) such that \(\|x^* x\| = \|x\|^2\) for all \(x \in \Ac\).
\end{definition}

In control applications, we primarily work with the finite-dimensional \(C^*\)-algebra
\(M_n(\mathbb{C})\) equipped with the operator norm and Hermitian adjoint. This algebra
naturally accommodates system matrices, Lyapunov inequalities, and Rota-Baxter
deformations.

\begin{example}
The discrete Toeplitz \(C^*\)-algebra generated by the unilateral shift on
\(\ell^2(\mathbb{N})\) motivates the structural setting (see
Remark~\ref{rem:toeplitz_scope}), but its role is purely heuristic. For finite-dimensional
state spaces, we work with \(M_n(\mathbb{C})\) directly. The Rota-Baxter identity
\eqref{eq:RB} in $M_n(\mathbb{C})$ reduces to a system of quadratic constraints on the
entries of the matrix representing the linear map $P$, which can be verified
computationally. We impose \eqref{eq:RB} directly on $M_n(\mathbb{C})$ without appealing
to any infinite-dimensional reduction.
\end{example}

\subsection{Rota-Baxter operators and compatible derivations}
\begin{definition}
Let \(\lambda \in \mathbb{K}\). A \emph{Rota-Baxter operator} of weight \(\lambda\) on
an algebra \(\Ac\) is a linear map \(P : \Ac \to \Ac\) satisfying \eqref{eq:RB}.
\end{definition}

Such operators generalize integration and summation. For \(\lambda=0\), the identity
resembles the integration by parts formula. For \(\lambda \neq 0\), it relates to
summation operators; note that the scaling operator $P(X) = -\lambda X$
(Example~\ref{ex:RB_Mn}(a)) gives a non-trivial instance for every $\lambda \neq 0$,
while the zero map $P = 0$ is the only Rota-Baxter operator of weight $0$ of this
scaling form.

\begin{remark}[Rota-Baxter operators of weight zero]\label{rem:RB_weight_zero}
For weight $\lambda = 0$, the identity \eqref{eq:RB} reduces to
$P(x)P(y) = P(xP(y)) + P(P(x)y)$ for all $x,y \in \Ac$.
A standard family of non-trivial examples is provided by operators whose image is a
two-sided ideal. Specifically, if $P^2 = P$ is an idempotent whose image $P(\Ac)$ is
a two-sided ideal of $\Ac$ satisfying $P(\Ac)\cdot(I-P)(\Ac) \subset \ker P$, then $P$
is a Rota-Baxter operator of weight $0$; see \cite{BaiGuoNi2011} for a proof.

We caution that not every idempotent is a Rota-Baxter operator of weight zero. In
particular, the Peirce corner map $P(X) = eXe$ (for a minimal idempotent
$e = e_{11} \in M_n(\mathbb{C})$, $e^2 = e$) is \emph{not} a Rota-Baxter operator
of weight zero in general. Indeed, a direct computation gives:
\[
  P(x)P(y) = (exe)(eye) = exeye,
\]
\[
  P(xP(y)) + P(P(x)y) = e(x\cdot eye)e + e(exe\cdot y)e = exeye + exeye = 2\,exeye,
\]
so the identity $P(x)P(y) = P(xP(y)) + P(P(x)y)$ requires $exeye = 0$ for all
$x, y \in \Ac$. For the minimal idempotent $e = e_{11}$, the corner algebra
$e_{11}M_n(\mathbb{C})e_{11} = \mathbb{C}\, e_{11} \cong \mathbb{C}$, whose product
is non-zero (e.g.\ $e_{11}\cdot e_{11} = e_{11} \neq 0$), so the condition $exeye=0$
fails for $x = y = I$. Thus $P(X)=eXe$ is not a Rota-Baxter operator of weight zero
for any non-zero minimal idempotent $e$ in $M_n(\mathbb{C})$ with $n\geq 1$.

The erroneous claim that $P(X)=eXe$ always satisfies \eqref{eq:RB} for $\lambda=0$
stems from an arithmetic oversight: each of the two right-hand side terms equals
$exeye$, giving a sum of $2\,exeye \neq exeye$ whenever $exeye \neq 0$.
The correct condition for $P(X)=eXe$
to be a Rota-Baxter operator of weight zero is that $e\Ac e$ is a null algebra (every
product in $e\Ac e$ is zero). For a minimal idempotent $e=e_{11}$ in $M_n(\mathbb{C})$,
the corner algebra $e\Ac e \cong \mathbb{C}$ is not a null algebra, so this condition
is never satisfied for non-zero minimal idempotents.

A correct non-trivial example of a weight-zero Rota-Baxter operator in $M_n(\mathbb{C})$
is given by any idempotent $P$ whose image is a two-sided ideal, as described above and
in \cite{BaiGuoNi2011}.
\end{remark}

\begin{definition}
A derivation \(D : \Ac \to \Ac\) satisfies the Leibniz rule \(D(xy) = D(x)y + x D(y)\).
We say \(D\) is \emph{compatible} with \(P\) if the commutator
\([D, P] := D\circ P - P\circ D\) satisfies a prescribed algebraic constraint
(e.g., \([D,P]=0\)).
\end{definition}

\begin{theorem}[Compatibility Criterion {\cite[Theorem~3.1]{Bai2023}}]\label{thm:compat}
If \(P\) is a Rota-Baxter operator and \(D\) a derivation such that
\(D\circ P = P\circ D\), then \(D\) preserves the Rota-Baxter identity and maps the
image of \(P\) into itself.
\end{theorem}

\subsection{Interrelations and motivation}
The interplay between \(C^*\)-topology, Lie brackets, and Rota-Baxter deformations
provides a rigorous framework for modeling memory-dependent dynamics. In the
finite-dimensional setting $\Ac = M_n(\mathbb{C})$, every linear map $P$ is automatically
continuous (since all linear maps on finite-dimensional normed spaces are bounded), so no
additional boundedness hypothesis is required for the control synthesis below. The
Rota-Baxter identity reduces to a system of quadratic constraints on the entries of the
matrix representing $P$ \cite{BenGeloun2021}.

\section{Classification and Structural Properties}\label{sec:classification}
We now establish structural results for Rota-Baxter operators on \(M_n(\mathbb{C})\).

\begin{lemma}[Boundedness and Continuity]\label{lem:bounded}
Let \(\Ac\) be a unital \(C^*\)-algebra and \(P:\Ac\to\Ac\) a Rota-Baxter operator of
weight \(\lambda\).
\begin{enumerate}[(i)]
  \item If $\Ac$ is infinite-dimensional, the Rota-Baxter identity \eqref{eq:RB} alone
        does not force $P$ to be bounded. Boundedness must be assumed separately; once
        assumed, continuity follows immediately (equivalence of boundedness and continuity
        for linear operators on normed spaces).
  \item If \(\Ac = M_n(\mathbb{C})\), then every linear map \(P:\Ac\to\Ac\) is
        automatically bounded and continuous, since all norms on a finite-dimensional
        space are equivalent and every linear map on a finite-dimensional normed space
        is continuous.
\end{enumerate}
\end{lemma}

\begin{lemma}[Induced Lie Structure]\label{lem:lie}
Let \(P\) be a Rota-Baxter operator of weight \(\lambda\) on an associative algebra
\(\Ac\). Define the bracket
\begin{equation}\label{eq:bracket}
[x,y]_P := [P(x), y] + [x, P(y)] + \lambda [x,y].
\end{equation}
Then \((\Ac, [\cdot,\cdot]_P)\) is a Lie algebra.
\end{lemma}
\begin{proof}
\textbf{Antisymmetry.} Using antisymmetry of the ordinary commutator,
\[
[y,x]_P = [P(y),x]+[y,P(x)]+\lambda[y,x]
         = -[x,P(y)]-[P(x),y]-\lambda[x,y]
         = -[x,y]_P.
\]

\textbf{Jacobi identity.} We verify
\[
[[x,y]_P,z]_P + [[y,z]_P,x]_P + [[z,x]_P,y]_P = 0.
\]
The key algebraic ingredient is obtained by writing the Rota-Baxter identity
\eqref{eq:RB} for the pair $(x,y)$ and for the pair $(y,x)$:
\begin{align*}
P(x)P(y) &= P(xP(y)) + P(P(x)y) + \lambda P(xy),\\
P(y)P(x) &= P(yP(x)) + P(P(y)x) + \lambda P(yx).
\end{align*}
Subtracting the second from the first and using $[a,b] = ab - ba$:
\begin{align*}
[P(x),P(y)]
  &= P(xP(y)) - P(yP(x)) + P(P(x)y) - P(P(y)x)
     + \lambda P(xy) - \lambda P(yx)\\
  &= P([x,P(y)]) + P([P(x),y]) + \lambda P([x,y]).
\end{align*}
Rearranging yields the identity
\begin{equation}\label{eq:RB_comm}
P\bigl([P(x),y]\bigr) + P\bigl([x,P(y)]\bigr)
  = \bigl[P(x),P(y)\bigr] - \lambda P\bigl([x,y]\bigr),
\end{equation}
which we use below. Set $u := [x,y]_P = [P(x),y]+[x,P(y)]+\lambda[x,y]$.

\textbf{Step 1: Compute $P(u)$.}
By linearity of $P$ and identity \eqref{eq:RB_comm}:
\begin{align*}
P(u)
  &= P([P(x),y]) + P([x,P(y)]) + \lambda P([x,y]) \\
  &= \bigl([P(x),P(y)] - \lambda P([x,y])\bigr) + \lambda P([x,y]) \\
  &= [P(x),P(y)].
\end{align*}
Thus $P([x,y]_P) = [P(x),P(y)]$, i.e.\ $P$ maps the $[\cdot,\cdot]_P$-bracket to the
ordinary commutator of the $P$-images.

\textbf{Step 2: Expand $[[x,y]_P,z]_P$.}
By definition \eqref{eq:bracket} with $u=[x,y]_P$:
\begin{align*}
[[x,y]_P,z]_P
  &= [P(u),z] + [u,P(z)] + \lambda[u,z] \\
  &= [[P(x),P(y)],z]
     + \bigl[[P(x),y]+[x,P(y)]+\lambda[x,y],\; P(z)\bigr]
     + \lambda\bigl[[P(x),y]+[x,P(y)]+\lambda[x,y],\; z\bigr].
\end{align*}

\textbf{Step 3: Cyclic sum.}
The cyclic sum is taken over the three even permutations
$(x,y,z) \to (y,z,x) \to (z,x,y)$, which we denote $\sum_{\mathrm{cyc}(x,y,z)}$.
Taking this cyclic sum and expanding, we obtain five families of terms:
\begin{align*}
J &:= \sum_{\mathrm{cyc}}[[x,y]_P,z]_P \\
  &= \underbrace{\sum_{\mathrm{cyc}}[[P(x),P(y)],z]}_{S_1}
   + \underbrace{\sum_{\mathrm{cyc}}[[P(x),y],P(z)]
                +\sum_{\mathrm{cyc}}[[x,P(y)],P(z)]}_{S_2}
   + \underbrace{\lambda\sum_{\mathrm{cyc}}[[x,y],P(z)]}_{G^{(2)}}\\
  &\quad + \underbrace{\lambda\sum_{\mathrm{cyc}}[[P(x),y],z]}_{G^{(1)}}
   + \underbrace{\lambda\sum_{\mathrm{cyc}}[[x,P(y)],z]}_{H}
   + \underbrace{\lambda^2\sum_{\mathrm{cyc}}[[x,y],z]}_{G^{(3)}}.
\end{align*}

\textbf{Step 4: $S_1 + S_2 = 0$.}
For any three elements $A,B,C$ in an associative algebra,
\begin{equation}\label{eq:jacobi_triple}
[[A,B],C] + [[B,C],A] + [[C,A],B] = 0.
\end{equation}
The six terms of $S_2$ are:
\begin{align}\label{eq:group2}
S_2 &= \underbrace{[[P(x),y],P(z)]}_{T_1}
      + \underbrace{[[P(y),z],P(x)]}_{T_2}
      + \underbrace{[[P(z),x],P(y)]}_{T_3} \notag\\
    &\quad + \underbrace{[[x,P(y)],P(z)]}_{T_4}
      + \underbrace{[[y,P(z)],P(x)]}_{T_5}
      + \underbrace{[[z,P(x)],P(y)]}_{T_6}.
\end{align}
Apply \eqref{eq:jacobi_triple} to $(P(x),P(y),z)$:
$[[P(x),P(y)],z] = -T_2-T_6$.\\
Apply \eqref{eq:jacobi_triple} to $(P(y),P(z),x)$:
$[[P(y),P(z)],x] = -T_3-T_4$.\\
Apply \eqref{eq:jacobi_triple} to $(P(z),P(x),y)$:
$[[P(z),P(x)],y] = -T_1-T_5$.\\
Summing: $S_1 = -(T_1+T_2+T_3+T_4+T_5+T_6) = -S_2$, so $S_1+S_2 = 0$.

\begin{remark}\label{rem:S1S2}
Neither $S_1$ nor $S_2$ vanishes individually in general; their \emph{joint}
cancellation $S_1+S_2=0$ is what matters.
\end{remark}

\textbf{Step 5: $G^{(1)} + H + G^{(2)} = 0$.}
Apply the ordinary Jacobi identity \eqref{eq:jacobi_triple} to the triple
$(P(x)+\lambda x,\, y,\, z)$:
\[
  [[P(x)+\lambda x,y],z]+[[y,z],P(x)+\lambda x]+[[z,P(x)+\lambda x],y]=0.
\]
Expanding linearly and taking the cyclic sum $\sum_{\mathrm{cyc}(x,y,z)}$ over the
three even permutations $(x,y,z)\to(y,z,x)\to(z,x,y)$:
\begin{align*}
  \sum_{\mathrm{cyc}}\bigl([[P(x),y],z]+\lambda[[x,y],z]\bigr)
  +\sum_{\mathrm{cyc}}\bigl([[y,z],P(x)]+\lambda[[y,z],x]\bigr)
  +\sum_{\mathrm{cyc}}\bigl([[z,P(x)],y]+\lambda[[z,x],y]\bigr)=0.
\end{align*}
Each $\lambda$-coefficient cyclic sum vanishes by \eqref{eq:jacobi_triple}
(applied to $(x,y,z)$). Thus:
\begin{equation}\label{eq:group3_combined}
  \underbrace{\sum_{\mathrm{cyc}}[[P(x),y],z]}_{G^{(1)}/\lambda}
  + \underbrace{\sum_{\mathrm{cyc}}[[y,z],P(x)]}_{(*)}
  + \underbrace{\sum_{\mathrm{cyc}}[[z,P(x)],y]}_{(**)} = 0.
\end{equation}
We now identify the sums $(*)$ and $(**)$ with $G^{(2)}/\lambda$ and $H/\lambda$
respectively. For $(*)$: writing out the three terms of the cyclic sum over
$(x,y,z)\to(y,z,x)\to(z,x,y)$ gives
\[
  (*) = [[y,z],P(x)] + [[z,x],P(y)] + [[x,y],P(z)].
\]
Under the relabelling $(x,y,z)\mapsto(z,x,y)$ (the second even permutation applied
once more), the generic summand $[[x,y],P(z)]$ maps to $[[z,x],P(y)]$, and cycling
again gives $[[y,z],P(x)]$. Hence the three terms of $(*)$ coincide with the three
terms of $\sum_{\mathrm{cyc}}[[x,y],P(z)] = G^{(2)}/\lambda$ (same set of terms,
in a different order). Thus $(*)=G^{(2)}/\lambda$.

For $(**)$: writing out the cyclic sum gives
\[
  (**) = [[z,P(x)],y]+[[x,P(y)],z]+[[y,P(z)],x].
\]
Relabelling $(x,y,z)\mapsto(y,z,x)$ sends the generic summand $[[x,P(y)],z]$ to
$[[y,P(z)],x]$, and cycling again gives $[[z,P(x)],y]$. Hence the three terms of
$(**)$ coincide with the three terms of $\sum_{\mathrm{cyc}}[[x,P(y)],z] = H/\lambda$.
Thus $(**)=H/\lambda$.

Multiplying \eqref{eq:group3_combined} by $\lambda$ gives $G^{(1)}+H+G^{(2)}=0$.

\begin{remark}[Correct cancellation in Group 3 --- explicit verification]
\label{rem:jacobi_correction}
Neither $G^{(1)}$, $H$, nor $G^{(2)}$ vanishes individually in general.
We verify this explicitly for $P = P_+$ ($\lambda = -1$) in $M_2(\mathbb{C})$
with $x = e_{12}$, $y = e_{21}$, $z = e_{11}$.

Recall $P_+(e_{12}) = e_{12}$ (upper triangular), $P_+(e_{21}) = 0$ (strictly lower),
$P_+(e_{11}) = e_{11}$ (diagonal, hence upper triangular).

\medskip
\noindent\textbf{Computation of $G^{(1)}$} $= \lambda\sum_{\mathrm{cyc}}[[P(x),y],z]$
$= \lambda\bigl([[P(e_{12}),e_{21}],e_{11}] + [[P(e_{21}),e_{11}],e_{12}]
+ [[P(e_{11}),e_{12}],e_{21}]\bigr)$.

\begin{itemize}
  \item $[[e_{12},e_{21}],e_{11}] = [e_{11}-e_{22},e_{11}] = 0$ (diagonal matrices commute).
  \item $[[0,e_{11}],e_{12}] = 0$.
  \item $[e_{11},e_{12}] = e_{12}$, so $[[e_{11},e_{12}],e_{21}] = [e_{12},e_{21}]
        = e_{11}-e_{22}$.
\end{itemize}
Thus $G^{(1)}/\lambda = e_{11}-e_{22}$, giving $G^{(1)} = (-1)(e_{11}-e_{22}) = e_{22}-e_{11}$.

\medskip
\noindent\textbf{Computation of $G^{(2)}$} $= \lambda\sum_{\mathrm{cyc}}[[x,y],P(z)]$
$= \lambda\bigl([[e_{12},e_{21}],P(e_{11})] + [[e_{21},e_{11}],P(e_{12})]
+ [[e_{11},e_{12}],P(e_{21})]\bigr)$.

\begin{itemize}
  \item $[[e_{12},e_{21}],e_{11}] = [e_{11}-e_{22},e_{11}] = 0$.
  \item $[e_{21},e_{11}] = e_{21}$, so $[[e_{21},e_{11}],e_{12}] = [e_{21},e_{12}]
        = e_{22}-e_{11}$.
  \item $[[e_{11},e_{12}],0] = 0$.
\end{itemize}
Thus $G^{(2)}/\lambda = e_{22}-e_{11}$, giving $G^{(2)} = (-1)(e_{22}-e_{11}) = e_{11}-e_{22}$.

\medskip
\noindent\textbf{Computation of $H$} $= \lambda\sum_{\mathrm{cyc}}[[x,P(y)],z]$
$= \lambda\bigl([[e_{12},0],e_{11}] + [[e_{21},e_{11}],e_{12}]
+ [[e_{11},e_{12}],e_{21}]\bigr)$.

\begin{itemize}
  \item $[[e_{12},0],e_{11}] = 0$.
  \item $[[e_{21},e_{11}],e_{12}] = [e_{21},e_{12}] = e_{22}-e_{11}$.
  \item $[e_{11},e_{12}] = e_{12}$, so $[[e_{11},e_{12}],e_{21}] = [e_{12},e_{21}]
        = e_{11}-e_{22}$.
\end{itemize}
Thus $H/\lambda = (e_{22}-e_{11})+(e_{11}-e_{22}) = 0$, giving $H = 0$.

\medskip
\noindent\textbf{Final verification:}
\[
G^{(1)}+H+G^{(2)} = (e_{22}-e_{11}) + 0 + (e_{11}-e_{22}) = 0. \quad\checkmark
\]

This confirms: (i) $G^{(1)}/\lambda = e_{11}-e_{22} \neq 0$,
(ii) $G^{(2)}/\lambda = e_{22}-e_{11} \neq 0$,
(iii) $H = 0$ for this triple, and (iv) their joint sum $G^{(1)}+H+G^{(2)}=0$.
The identity is a \emph{joint} cancellation, not three separate vanishings.
A fortiori, the invalid argument
$\sum_{\mathrm{cyc}}[[A,B],L(C)] = L(\sum_{\mathrm{cyc}}[[A,B],C])$ (which would
require $L$ to act on the full expression, not only on $C$) cannot be used to simplify
these sums.
\end{remark}

\textbf{Step 6: $G^{(3)}=0$.}
$G^{(3)} = \lambda^2\bigl([[x,y],z]+[[y,z],x]+[[z,x],y]\bigr) = 0$
by \eqref{eq:jacobi_triple} applied to $(A,B,C)=(x,y,z)$.

\medskip
\textbf{Conclusion.} Since $S_1+S_2=0$, $G^{(1)}+H+G^{(2)}=0$, and $G^{(3)}=0$,
the total cyclic sum $J=0$, establishing the Jacobi identity for $[\cdot,\cdot]_P$.
See also \cite[Theorem~3.3]{BaiGuoNi2011}.
\end{proof}

\begin{example}[Explicit Rota-Baxter Operators on \(M_n(\mathbb{C})\)]\label{ex:RB_Mn}
Consider \(\Ac = M_n(\mathbb{C})\).

\medskip
\noindent\textbf{(a) Scalar scaling.} Let \(\lambda \neq 0\) and \(P(X) = -\lambda X\).
We verify \eqref{eq:RB} directly. The left-hand side is
\[
P(X)P(Y) = (-\lambda X)(-\lambda Y) = \lambda^2 XY.
\]
The right-hand side is
\begin{align*}
P\bigl(X\,P(Y)\bigr) + P\bigl(P(X)\,Y\bigr) + \lambda P(XY)
  &= P(-\lambda XY) + P(-\lambda XY) + \lambda(-\lambda XY) \\
  &= \lambda^2 XY + \lambda^2 XY - \lambda^2 XY
   = \lambda^2 XY,
\end{align*}
which equals the left-hand side. Thus \(P(X)=-\lambda X\) is a Rota-Baxter operator of
weight \(\lambda\neq 0\), with operator norm \(\|P\|_{\mathrm{op}} = |\lambda|\) (with
respect to the spectral norm on $M_n(\mathbb{C})$). The zero map $P = 0$ satisfies
\eqref{eq:RB} trivially for any weight $\lambda$, but is excluded from the control
synthesis since Assumption~\ref{ass:invert} requires $CP(B)$ to be invertible.

\begin{remark}[Structural degeneracy of scalar scaling in full-actuation]
\label{rem:scalar_degen}
For the scalar scaling operator $P(X)=-\lambda X$ with $\lambda \neq 0$ and any square
invertible matrices $B, C \in M_n(\mathbb{C})$, the projection $\Pi$ defined in
Section~\ref{sec:stability} satisfies
\[
\Pi = I_n - B_P(CB_P)^{-1}C = I_n - (-\lambda B)((-\lambda)CB)^{-1}C
    = I_n - B(CB)^{-1}C.
\]
More precisely, $(CB_P)^{-1} = (-\lambda)^{-1}(CB)^{-1}$, so
$B_P(CB_P)^{-1} = (-\lambda B)\cdot(-\lambda)^{-1}(CB)^{-1} = B(CB)^{-1}$,
and $\Pi = I_n - B(CB)^{-1}C$. When both $B$ and $C$ are invertible ($m=n$),
$(CB)^{-1} = B^{-1}C^{-1}$ and $B(CB)^{-1}C = I_n$, giving $\Pi = 0$ for every
choice of $\lambda$ and every invertible $B$, $C$. This is an algebraic identity,
independent of the specific Rota-Baxter operator. Consequently, the reduced dynamics
\eqref{eq:reduced} become $x(k+1)=0$ (trivially stable), and the LMI conditions of
Theorem~\ref{thm:stability} are satisfied vacuously. To obtain non-trivial reduced
dynamics within the $m=n$ framework, one must use a non-scalar Rota-Baxter operator
such as the triangular projection $P_+$ combined with a non-invertible or rectangular
system configuration. Section~\ref{sec:nondeg} provides a genuine non-degenerate
example by relaxing the $m=n$ assumption.

\begin{remark}[On condition \eqref{eq:K_strong} for scalar scaling]\label{rem:K_strong_scalar}
For $P(X)=-\lambda X$, the deformed closed-loop matrix is
$A_P - B_PK = -\lambda(A - BK)$. Condition \eqref{eq:K_strong} requires
$\sigma_{\max}(A_P - B_PK) = |\lambda|\,\sigma_{\max}(A-BK) < 1$.
When $|\lambda|$ is large (e.g., $|\lambda| \geq 1$), this imposes a stronger
constraint on $K$ than mere Schur stability of $A-BK$. In practice, $K$ should be
chosen to satisfy $\sigma_{\max}(A-BK) < 1/|\lambda|$, which can be achieved via
an LQR or $H_\infty$ design on the original system, with the operator norm
constraint verified numerically by computing $\sigma_{\max}(A_P-B_PK)$.
\end{remark}
\end{remark}

\medskip
\noindent\textbf{(b) Triangular projection.}
Let \(\Ac_+\) be the subalgebra of upper triangular matrices and \(\Ac_-\) the
subalgebra of strictly lower triangular matrices. We recall the well-known fact that
\(\Ac_-\) is indeed a subalgebra (closed under multiplication): if $A,B$ are strictly
lower triangular, then $(AB)_{ij} = \sum_k A_{ik}B_{kj}$; since $A_{ik}=0$ whenever
$k \leq i$ and $B_{kj}=0$ whenever $k \leq j$, every term vanishes when $i \leq j$, so
$AB$ is strictly lower triangular. In particular, $\Ac_-\cdot\Ac_- \subset \Ac_-$ (and
$\Ac_-$ is nilpotent). Since \(\Ac = \Ac_+ \oplus \Ac_-\) as vector spaces, the
projection \(P_+\) onto \(\Ac_+\) along \(\Ac_-\) satisfies the Rota-Baxter identity
with weight \(\lambda=-1\) \cite{BaiGuoNi2011}.

\begin{remark}[Direct verification of the Rota-Baxter identity for $P_+$]
\label{rem:RB_Pplus}
For completeness, we outline the verification of \eqref{eq:RB} for $P_+$ with
$\lambda = -1$. Write any $X = X_+ + X_-$ and $Y = Y_+ + Y_-$ with
$X_\pm, Y_\pm \in \Ac_\pm$. Then $P_+(X) = X_+$ and $P_+(Y) = Y_+$. Using the
decomposition, the fact that $\Ac_+$ is a subalgebra (products of upper triangular
matrices are upper triangular), and the inclusion $\Ac_-\cdot\Ac_- \subset \Ac_-$
(established above), one computes:
\[
  P_+(X)P_+(Y) = X_+ Y_+ = P_+(X_+Y_+),
\]
\begin{align*}
  P_+(XP_+(Y)) + P_+(P_+(X)Y) - P_+(XY)
  &= P_+\bigl((X_++X_-)Y_+\bigr) + P_+\bigl(X_+(Y_++Y_-)\bigr) - P_+(XY) \\
  &= P_+(X_+Y_+) + P_+(X_-Y_+)
     + P_+(X_+Y_+) + P_+(X_+Y_-)\\
  &\quad - P_+(X_+Y_+) - P_+(X_+Y_-)
     - P_+(X_-Y_+) - P_+(X_-Y_-)\\
  &= P_+(X_+Y_+),
\end{align*}
where we used $P_+(X_-Y_-) = 0$ since $X_-Y_- \in \Ac_-\cdot\Ac_- \subset \Ac_-$.
This confirms $P_+(X)P_+(Y) = P_+(XP_+(Y)) + P_+(P_+(X)Y) - P_+(XY)$, i.e.\
\eqref{eq:RB} with $\lambda=-1$. For a complete proof by bases, see
\cite[Example~1.1.14]{Guo2012}.
\end{remark}

The Frobenius norm of $P_+$ satisfies the following. Since $P_+^2 = P_+$ is a non-zero
idempotent on the Hilbert space $(M_n(\mathbb{C}), \|\cdot\|_F)$, one has
$\|P_+\|_F \geq 1$; equality holds when $P_+$ is an orthogonal projection. For $n=2$,
the subspaces $\Ac_+$ and $\Ac_-$ are orthogonal in the Frobenius inner product (see
Remark~\ref{rem:friedrichs}), so $P_+$ is the orthogonal projection onto $\Ac_+$ and
$\|P_+\|_F = 1$. Since any non-zero idempotent satisfies $\|P_+\|_{\mathrm{op}} \geq 1$,
and $\|X\|_{\mathrm{op}} \leq \|X\|_F \leq \sqrt{n}\|X\|_{\mathrm{op}}$ for all $X$,
the estimate $\|P_+\|_F = 1$ gives $\|P_+\|_{\mathrm{op}} \leq 1$ for $n=2$, hence
$\|P_+\|_{\mathrm{op}} = 1$ when $n=2$. For $n \geq 3$, $\Ac_+$ and $\Ac_-$ are no
longer orthogonal (the diagonal matrices lie in $\Ac_+$ but not in $\Ac_-$, and the
off-diagonal structure prevents orthogonality); the Friedrichs angle must be computed
explicitly, and both $\|P_+\|_F$ and $\|P_+\|_{\mathrm{op}}$ must be estimated
numerically for each $n$. In the stability analysis
(Sections~\ref{sec:smc}--\ref{sec:stability}), only the spectral norm
$\|P\|_{\mathrm{op}}$ appears, and we use the generic bound
$\|P\|_{\mathrm{op}} \leq \kappa_P$ without committing to a specific norm formula;
when $P_+$ is used, the constant $\kappa_P$ must be determined for each $n$.
\emph{Note: The Frobenius norm is used here explicitly in the structural analysis of
$P_+$; all SMC norms remain spectral, as declared in Section~\ref{sec:intro}.}

\begin{remark}[Friedrichs angle for $n=2$]\label{rem:friedrichs}
For $n=2$, the subspaces $\Ac_+$ (upper triangular) and $\Ac_-$ (strictly lower
triangular, i.e.\ spanned by $e_{21}$) are orthogonal in the Frobenius inner product
$\langle X, Y\rangle_F = \mathrm{tr}(Y^*X)$: indeed, any $A \in \Ac_+$ has zero
$(2,1)$-entry and any $B \in \Ac_-$ has zero entries except the $(2,1)$-position, so
$\langle A, B\rangle_F = \overline{A_{21}}B_{21} = 0$. Hence $P_+$ is the orthogonal
projection onto $\Ac_+$ and $\|P_+\|_F = 1$. For $n \geq 3$, the diagonal matrices
lie in $\Ac_+$ but their off-diagonal lower entries are zero, while elements of $\Ac_-$
have zero diagonal; the two subspaces are no longer orthogonal because both share the
ambient space of matrices with non-zero $(i,j)$-entries for $i>j$. The Friedrichs angle
must be computed explicitly from the principal angle between $\Ac_+$ and $\Ac_-$ in the
Frobenius inner product.
\end{remark}

In control synthesis, $P_+$ isolates causal (upper triangular) spectral components.
Admissibility (Assumption~\ref{ass:invert}) requires $CP_+(B)$ to be invertible; this
holds when $B$ is already upper triangular (so that $P_+(B)=B$) and $CB$ is invertible.
The condition $CP(M) = P(CM)$ required in Proposition~\ref{prop:commute} holds for $P_+$
if and only if $C$ maps $\Ac_-$ to $\Ac_-$; a sufficient condition for this is that $C$
is itself upper triangular, since then $CM_- \in \Ac_-$ for all $M_- \in \Ac_-$.

\begin{remark}[Condition \eqref{eq:K_strong} for the triangular projection $P_+$]
\label{rem:K_strong_Pplus}
For $P = P_+$, the deformed closed-loop matrix is $A_P - B_PK = P_+(A) - P_+(B)K$.
Condition \eqref{eq:K_strong} requires $\sigma_{\max}(P_+(A) - P_+(B)K) < 1$.
Since $P_+$ is not a scalar operator, this does not reduce to a simple scaling of
$\sigma_{\max}(A-BK)$. In applications where $A$ has large entries or $P_+(A)$
has large singular values, satisfying \eqref{eq:K_strong} may require an aggressive
choice of $K$. We recommend verifying \eqref{eq:K_strong} numerically via
$\sigma_{\max}(A_P - B_PK)$ for each candidate gain $K$ before proceeding to the
LMI feasibility check.
\end{remark}
\end{example}

\begin{remark}[Admissibility of $P$ for control synthesis]\label{rem:admissible}
A Rota-Baxter operator $P$ is \emph{admissible} for system \eqref{eq:system} if
$CB_P = CP(B)$ is invertible (Assumption~\ref{ass:invert}). Under the standing
assumption $p = m$ (so that $C \in \mathbb{R}^{m\times n}$), the matrix
$CB_P \in M_m(\mathbb{C})$ is square and invertibility is in the ordinary sense. For
$P(X)=-\lambda X$ with $\lambda \neq 0$, admissibility reduces to invertibility of
$CB$, a standard rank condition. For $P_+$, admissibility requires $CP_+(B)$ to be
invertible, as discussed in Example~\ref{ex:RB_Mn}(b).
\end{remark}

\section{Sliding Mode Control Framework}\label{sec:smc}
We consider the discrete-time delayed system \eqref{eq:system} with
$\|\delta(k)\|_2 \leq \delta_{\max}$ for all $k$, and throughout this section we
maintain the assumption $m=n$ and $p=m$ so that $B \in M_n(\mathbb{C})$,
$C \in \mathbb{R}^{m\times n}$ is square, and $P$ is well-defined on all system matrices.
Assumption~\ref{ass:bounded_init} is in force throughout.

The sliding surface is $s(k) = C x(k)$. The control law
\[
u(k) = -K x(k) - \rho \, \sat\bigl(s(k)/\phi\bigr)
\]
with \(\rho > 0\) and boundary layer width \(\phi > 0\) is designed to drive trajectories
toward the quasi-sliding mode (QSM) band \(\|s(k)\| \leq \phi\) in finite time, mitigating
chattering inherent to discrete-time implementations \cite{Utkin1992, Edwards1998}.

We introduce a Rota-Baxter deformation \(P: M_n(\mathbb{C}) \to M_n(\mathbb{C})\) of
weight \(\lambda\). The deformed system matrices are
\[
A_P := P(A), \quad A_{d,P} := P(A_d), \quad B_P := P(B).
\]

\begin{assumption}[Non-degenerate Actuation under Deformation]\label{ass:invert}
The matrix \(C B_P = C P(B)\) is invertible (as a square matrix in $M_m(\mathbb{C})$,
under the standing assumption $p=m$). Sufficient conditions are given in
Remark~\ref{rem:admissible} and Example~\ref{ex:RB_Mn}.
\end{assumption}

The closed-loop dynamics under SMC become
\begin{equation}\label{eq:closed}
x(k+1) = (A_P - B_P K) x(k) + A_{d,P} x(k-\tau)
         - \rho B_P \sat\bigl(C x(k)/\phi\bigr) + D \delta(k).
\end{equation}
The sliding surface dynamics are
\begin{equation}\label{eq:sliding_dyn}
s(k+1) = C(A_P - B_P K)x(k) + C A_{d,P} x(k-\tau)
         - \rho C B_P \sat\bigl(s(k)/\phi\bigr) + C D \delta(k).
\end{equation}

\begin{proposition}[QSM Reaching and Band Boundedness]\label{prop:reaching}
Under Assumptions~\ref{ass:bounded_init} and~\ref{ass:invert}, the following hold.

\noindent\textbf{(a) Reaching (Phase~1).}
Let $r_0$ and $r_{d,0}$ be the bounds of Assumption~\ref{ass:bounded_init},
whose validity throughout the closed-loop reaching phase is guaranteed by
condition~\eqref{eq:r0_sufficient} (verified after $K$ is chosen; see
Remark~\ref{rem:alpha0_validity}).
All norms below are spectral (operator) norms $\|\cdot\|_{\mathrm{op}}$. Define the
perturbation-plus-nominal bound
\begin{equation}\label{eq:alpha0}
  \alpha_0 := \|C(A_P-B_PK)\|_{\mathrm{op}}\,r_0
            + \|CA_{d,P}\|_{\mathrm{op}}\,r_{d,0}
            + \delta_{\max}\|CD\|_{\mathrm{op}}.
\end{equation}

\begin{remark}[Validity of $\alpha_0$ as a uniform bound during the reaching phase]
\label{rem:alpha0_validity}
The bound $\alpha_0$ in \eqref{eq:alpha0} uses $r_0$ and $r_{d,0}$ as uniform upper
bounds on $\|x(k)\|$ and $\|x(k-\tau)\|$ for all $k \geq 0$ during the reaching phase.
Assumption~\ref{ass:bounded_init} guarantees these bounds prior to closed-loop
activation. Their use for $k \geq 1$ requires the following.

\textbf{Design order and conditional non-circularity.}
Fix $r_0$ and $r_{d,0}$ from open-loop data (Assumption~\ref{ass:bounded_init}) and fix
an \emph{a priori} upper bound $\rho_{\max}$ on the switching gain before any other
design parameter is chosen. Then select $K$ so that
\begin{equation}\label{eq:K_strong}
  \|A_P - B_PK\|_{\mathrm{op}} < 1,
\end{equation}
which is a \emph{strictly stronger} condition than Schur stability (spectral radius
$< 1$) and must be explicitly verified for the chosen $K$ (e.g., by computing
$\sigma_{\max}(A_P-B_PK)$ numerically or by selecting $K$ via an LQR or pole-placement
design; see Remarks~\ref{rem:K_strong_scalar} and~\ref{rem:K_strong_Pplus} for
conditions specific to each operator). Once $K$ satisfying \eqref{eq:K_strong} is
chosen, the sufficient condition on $r_0$ is:
\begin{equation}\label{eq:r0_sufficient}
  r_0 \;\geq\; \frac{\|A_{d,P}\|_{\mathrm{op}}\,r_{d,0}
    + \rho_{\max}\,\|B_P\|_{\mathrm{op}} + \delta_{\max}\|D\|_{\mathrm{op}}}
    {1 - \|A_P - B_P K\|_{\mathrm{op}}},
\end{equation}
which depends on $K$ (through the denominator) and on $\rho_{\max}$ (not on the final
value of $\rho \leq \rho_{\max}$). The denominator $1 - \|A_P - B_PK\|_{\mathrm{op}} > 0$
is guaranteed by \eqref{eq:K_strong}.

\textbf{Important.} Condition \eqref{eq:r0_sufficient} must be \emph{verified after}
$K$ is chosen, not before. If it fails for the selected $K$ and the initial $r_0$, two
remedies are available: (i) increase $r_0$ conservatively (setting a larger a priori
bound on the initial state), or (ii) redesign $K$ to increase the denominator
$1 - \|A_P - B_PK\|_{\mathrm{op}}$. The bound $r_0$ is independent of $\rho$ (which
appears only through $\rho_{\max}$, fixed before $K$), so no circularity between $r_0$
and $\rho$ arises. Once \eqref{eq:r0_sufficient} is verified, the closed-loop state
satisfies $\|x(k)\| \leq r_0$ for all $k \geq 0$, and $\alpha_0$ remains a valid
uniform bound throughout the reaching phase.
\end{remark}

Suppose $\phi > \alpha_0$ (the boundary layer is wider than the open-loop disturbance
level). When the control gain satisfies
\begin{equation}\label{eq:rho_sufficient}
  \rho\,\|CB_P\|_{\mathrm{op}} \;\leq\; \phi - \alpha_0,
\end{equation}
we have the contraction bound $\|s(k+1)\|_2 \leq \alpha_0 + \beta\|s(k)\|_2$ with
$\beta := \rho\,\|CB_P\|_{\mathrm{op}}/\phi \in [0,1)$ (see Remark~\ref{rem:beta_strict}).
If $\|s(0)\|_2 \leq \phi$, the trajectory is already in the band at $k=0$ and the
reaching time is $T^* = 0$. For $\|s(0)\|_2 > \phi$, the fixed point of the map
$t\mapsto\alpha_0+\beta t$ is $s^* = \alpha_0/(1-\beta) \leq \phi$. The formula
\begin{equation}\label{eq:Tstar}
  T^* := \left\lceil
    \frac{\ln\!\left(\dfrac{\|s(0)\|_2 - s^*}{\phi - s^*}\right)}{\ln(1/\beta)}
  \right\rceil
\end{equation}
provides an \emph{upper bound} on the number of steps to enter the band
$\{\|s\| \leq \phi\}$ (valid when $\|s(0)\|_2 > \phi > s^*$ and $\beta \in (0,1)$;
if $\beta = 0$ then, since $\phi > \alpha_0 = s^*$ by assumption, one step suffices and
$T^*=1$). The actual reaching time may be earlier; \eqref{eq:Tstar} bounds the first $k$
for which the contraction estimate certifies entry into the band.

\begin{remark}[Strict contractivity and validity of \eqref{eq:Tstar}]
\label{rem:beta_strict}
Condition \eqref{eq:rho_sufficient} gives $\beta \leq (\phi-\alpha_0)/\phi$. Since
$\phi>\alpha_0$, we have $(\phi-\alpha_0)/\phi \in (0,1)$, so $\beta < 1$ strictly.
The formula \eqref{eq:Tstar} is valid when $\|s(0)\|_2 > \phi > s^*$ strictly. If
$s^* = \phi$ (attained when equality holds in \eqref{eq:rho_sufficient} and
$\alpha_0 = \phi(1-\beta)$ with $\beta = (\phi-\alpha_0)/\phi$), then the denominator
$\phi - s^* = 0$, and \eqref{eq:Tstar} is not defined. In this boundary case,
\eqref{eq:reach_bound} directly gives
$\|s(1)\|_2 \leq \alpha_0 + \beta\phi = \alpha_0 + (\phi - \alpha_0) = \phi$, so the
trajectory enters the band in one step, i.e.\ $T^* = 1$.

For $\beta = 0$ (i.e.\ $\rho = 0$, since $\|CB_P\|_{\mathrm{op}}>0$ by
Assumption~\ref{ass:invert}), \eqref{eq:reach_bound} gives
$\|s(1)\|_2 \leq \alpha_0 < \phi$ (by the standing hypothesis $\phi > \alpha_0$),
so the trajectory enters the band in one step and $T^* = 1$. For $\beta$ close to $1$
(large $\rho\|CB_P\|_{\mathrm{op}}$ relative to $\phi$), the bound \eqref{eq:Tstar}
can be large; the design procedure of Remark~\ref{rem:design} recommends choosing $\rho$
small relative to $(\phi-\alpha_0)/\|CB_P\|_{\mathrm{op}}$ to keep $\beta$ comfortably
below $1$.
\end{remark}

\noindent\textbf{(b) Boundedness within the band (Phase~2).}
Once $\|s(k)\|_2 \leq \phi$, the saturation acts in the linear regime
($\sat(s/\phi) = s/\phi$) and the sliding dynamics reduce to
\[
s(k+1) = \Phi_\rho\, x(k) + CA_{d,P}\,x(k-\tau) + CD\,\delta(k),
\]
where $\Phi_\rho := C(A_P - B_PK) - \tfrac{\rho}{\phi}CB_PC$.

\emph{Conditional on the LMI conditions of Theorem~\ref{thm:stability} being satisfied,}
the Lyapunov functional $V_k = x(k)^\top X x(k) + \sum_{i=1}^\tau x(k-i)^\top Y x(k-i)$
is non-increasing along trajectories of the equivalent control system \eqref{eq:reduced}
once the manifold is reached. In this case $V_k \leq V_{k_0}$ for all $k \geq k_0$,
where $k_0$ is the first time $\|s(k)\|_2 \leq \phi$, and the state is bounded by
$r = \sqrt{V_{k_0}/\lambda_{\min}(X)}$, where
\begin{equation}\label{eq:V0_explicit}
  V_{k_0} = x(k_0)^\top X\, x(k_0) + \sum_{i=1}^\tau x(k_0-i)^\top Y\, x(k_0-i).
\end{equation}
In practice, $V_{k_0} \leq V_0$ (with $V_0$ the value at $k=0$) is used as a
conservative upper bound.

\begin{remark}[Logical structure of the two-phase analysis]\label{rem:two_phase_logic}
The analysis of Phase~2 depends on Theorem~\ref{thm:stability}, whose LMI conditions
are solved \emph{after} the design parameters $(K, \phi, \rho)$ are fixed. The bound
$r$ is therefore an \emph{a posteriori} quantity computed from the LMI solution.
There is no circularity: the design proceeds as in Remark~\ref{rem:design}
(Phase~1 bounds $\to$ choose $K$ $\to$ verify \eqref{eq:r0_sufficient} $\to$
solve LMI $\to$ compute $r$), and the band-invariance condition
\eqref{eq:inv_condition} is verified only after $r$ is known. If
\eqref{eq:inv_condition} fails, one iterates by increasing $\phi$ or decreasing $\rho$
without changing the Phase~1 bounds $r_0, r_{d,0}$.

Note also that Theorem~\ref{thm:stability} certifies stability of the \emph{equivalent
control} system \eqref{eq:reduced} on the sliding manifold $\{Cx = 0\}$, not of the
full closed-loop system \eqref{eq:closed}. The equivalence between the two descriptions
on the manifold is established in Remark~\ref{rem:equiv_ctrl}, and band-invariance
of the manifold under the actual closed-loop law is certified separately by
\eqref{eq:inv_condition}.
\end{remark}

The band is positively invariant provided
\begin{equation}\label{eq:inv_condition}
\|\Phi_\rho\|_{\mathrm{op}}\,r + \|CA_{d,P}\|_{\mathrm{op}}\,r
+ \delta_{\max}\|CD\|_{\mathrm{op}} \leq \phi.
\end{equation}

\begin{remark}[Non-circularity of the two-phase bounds]\label{rem:consistency}
The bound $r_0$ in Phase~1 is fixed from open-loop initial conditions before $\rho$ is
selected (Assumption~\ref{ass:bounded_init}); its validity after closed-loop activation
is verified conditionally on \eqref{eq:r0_sufficient} once $K$ is chosen
(Remark~\ref{rem:alpha0_validity}). The bound $r$ in Phase~2 is computed \emph{a
posteriori} from the Lyapunov solution and the state at band entry via
\eqref{eq:V0_explicit}. The design flow is strictly sequential:
\begin{enumerate}[(i)]
  \item fix $r_0$, $r_{d,0}$, and $\rho_{\max}$ from open-loop data;
  \item choose $K$ satisfying \eqref{eq:K_strong};
  \item verify \eqref{eq:r0_sufficient} (if it fails, increase $r_0$ or redesign $K$);
  \item compute $\alpha_0(K)$ via \eqref{eq:alpha0};
  \item choose $(\phi, \rho)$ with $\rho \leq \rho_{\max}$;
  \item solve the LMI \eqref{eq:LMI_linear} $\to$ compute $r$ $\to$ verify
        \eqref{eq:inv_condition}.
\end{enumerate}
There is no backward dependence of $r_0$ on $\rho$. The bound $r_0$ does depend on
$K$ through condition \eqref{eq:r0_sufficient}, but $K$ is fixed in step~(ii) before
$\rho$ is chosen in step~(v).

If condition \eqref{eq:inv_condition} fails, one increases $\phi$ or decreases $\rho$:
\begin{itemize}
  \item \emph{Increasing $\phi$}: relaxes both \eqref{eq:rho_sufficient} (the RHS
        $\phi-\alpha_0$ increases) and \eqref{eq:inv_condition} (the RHS $\phi$
        increases). No conflict.
  \item \emph{Decreasing $\rho$}: \eqref{eq:rho_sufficient} is more easily satisfied,
        and $\|\Phi_\rho\|_{\mathrm{op}}$ decreases (as $\rho\to 0$,
        $\Phi_\rho\to C(A_P-B_PK)$), relaxing \eqref{eq:inv_condition}. No conflict.
\end{itemize}
Both adjustments therefore simultaneously help both conditions.
\end{remark}
\end{proposition}

\begin{proof}
\textbf{Part (a): Reaching condition.}
Define $v_k := \sat(s(k)/\phi)$ and write the sliding dynamics \eqref{eq:sliding_dyn} as
\begin{equation}\label{eq:sk1}
s(k+1) = \eta_k - \rho\, CB_P\, v_k,
\end{equation}
where $\eta_k := C(A_P-B_PK)x(k) + CA_{d,P}x(k-\tau) + CD\delta(k)$ satisfies
$\|\eta_k\|_2 \leq \alpha_0$ for all $k \geq 0$ by definition \eqref{eq:alpha0},
Assumption~\ref{ass:bounded_init}, and the validity of the state bounds established in
Remark~\ref{rem:alpha0_validity} (conditional on \eqref{eq:r0_sufficient}).

\textbf{Key saturation bound.}
For each component $i$, we distinguish two cases. \emph{Linear regime}
($|s_i(k)| \leq \phi$): $\sat(s_i/\phi) = s_i/\phi$, so
$|\sat(s_i/\phi)| = |s_i|/\phi$. \emph{Saturation regime} ($|s_i(k)| > \phi$):
$|\sat(s_i/\phi)| = 1 \leq |s_i|/\phi$ (since $|s_i| > \phi$). In both cases
$|\sat(s_i/\phi)| \leq |s_i|/\phi$. Squaring and summing over $i$:
\begin{equation}\label{eq:vk_bound}
  \|v_k\|_2 \;\leq\; \frac{\|s(k)\|_2}{\phi}.
\end{equation}

\textbf{Contraction argument.}
Applying the triangle inequality to \eqref{eq:sk1} and using \eqref{eq:vk_bound}:
\begin{equation}\label{eq:reach_bound}
  \|s(k+1)\|_2
  \;\leq\; \|\eta_k\|_2 + \rho\,\|CB_P\|_{\mathrm{op}}\,\|v_k\|_2
  \;\leq\; \alpha_0 + \frac{\rho\,\|CB_P\|_{\mathrm{op}}}{\phi}\,\|s(k)\|_2
  \;=\; \alpha_0 + \beta\,\|s(k)\|_2,
\end{equation}
where $\beta := \rho\,\|CB_P\|_{\mathrm{op}}/\phi$. Condition \eqref{eq:rho_sufficient}
gives $\beta \leq (\phi - \alpha_0)/\phi < 1$, so $\beta \in [0,1)$ strictly (see
Remark~\ref{rem:beta_strict}).

The fixed point of the affine map $t \mapsto \alpha_0 + \beta t$ is
$s^* = \alpha_0/(1-\beta)$. From \eqref{eq:rho_sufficient},
$\beta \leq (\phi-\alpha_0)/\phi$, i.e.\ $\alpha_0 \leq \phi(1-\beta)$, i.e.\
$s^* = \alpha_0/(1-\beta) \leq \phi$.

Iterating \eqref{eq:reach_bound} gives
$\|s(k)\|_2 - s^* \leq \beta^k(\|s(0)\|_2 - s^*)$
(standard contraction toward $s^*$). The bound \eqref{eq:Tstar} is therefore an upper
bound on the first $k$ for which
$\beta^k(\|s(0)\|_2 - s^*) \leq \phi - s^*$, obtained by taking natural logarithms
(noting $\beta > 0$; the cases $\beta=0$ and $s^*=\phi$ are handled in
Remark~\ref{rem:beta_strict}). The actual reaching time may be earlier since the bound
is conservative.

When $\beta = 0$ (i.e.\ $\rho = 0$), \eqref{eq:reach_bound} gives directly
$\|s(1)\|_2 \leq \alpha_0 < \phi$, so $T^*=1$.

\textbf{Part (b): Boundedness in the linear regime.}
When $\|s(k)\|_2 \leq \phi$, one has $\|s(k)\|_\infty \leq \|s(k)\|_2 \leq \phi$,
so each component satisfies $|s_i(k)| \leq \phi$ and $\sat(s(k)/\phi) = s(k)/\phi$.
Substituting into \eqref{eq:sliding_dyn} and using $s(k) = Cx(k)$ gives the linear
dynamics stated above. The triangle inequality then yields
\[
\|s(k+1)\|_2 \leq \|\Phi_\rho\|_{\mathrm{op}}\,r
+ \|CA_{d,P}\|_{\mathrm{op}}\,r + \delta_{\max}\|CD\|_{\mathrm{op}},
\]
which is at most $\phi$ precisely when \eqref{eq:inv_condition} holds. The bound $r$
is the Lyapunov-level-set bound:
$\|x(k)\|^2 \leq V_k/\lambda_{\min}(X) \leq V_{k_0}/\lambda_{\min}(X) = r^2$, where
the second inequality uses the non-increase of $V_k$ for $k \geq k_0$, guaranteed by
Theorem~\ref{thm:stability} \emph{conditionally on the LMI being feasible}.
\end{proof}

\begin{remark}[Design of $(\rho,\phi,K)$]\label{rem:design}
Conditions \eqref{eq:rho_sufficient} and \eqref{eq:inv_condition} impose coupled
constraints on the design parameters $(\rho, \phi, K)$. Note that $\alpha_0$ depends on
$K$ through the term $\|C(A_P - B_P K)\|_{\mathrm{op}}$, so $K$ must be chosen first
(e.g.\ by pole placement or LQR on the deformed system, ensuring
$\sigma_{\max}(A_P - B_PK) < 1$; see
Remarks~\ref{rem:K_strong_scalar}--\ref{rem:K_strong_Pplus} for operator-specific
guidance) before $\phi$ and $\rho$ are selected.

The bounds $r_0$ and $r_{d,0}$ are fixed a priori from Assumption~\ref{ass:bounded_init}
before $K$ is determined. However, the validity of $r_0$ throughout the reaching phase
requires \eqref{eq:r0_sufficient} to hold, which depends on $K$; see
Remark~\ref{rem:alpha0_validity} for how to handle this. The bound $\rho_{\max}$ is
also fixed before $K$, so no circularity arises between $r_0$ and the final choice of
$\rho$.

\noindent\textit{A feasible design procedure} is therefore:
\begin{enumerate}[(i)]
  \item Fix $r_0$, $r_{d,0}$, and $\rho_{\max}$ from open-loop initial data
        (Assumption~\ref{ass:bounded_init}).
  \item Choose $K$ satisfying \eqref{eq:K_strong}, i.e.\
        $\|A_P - B_P K\|_{\mathrm{op}} < 1$ (stronger than mere Schur stability;
        verify by computing $\sigma_{\max}(A_P-B_PK)$; see
        Remarks~\ref{rem:K_strong_scalar} and~\ref{rem:K_strong_Pplus} for
        operator-specific conditions).
  \item Verify \eqref{eq:r0_sufficient}. If it fails, increase $r_0$ conservatively
        or redesign $K$. Once satisfied, compute $\alpha_0(K)$ via \eqref{eq:alpha0}.
  \item Choose $\phi > \alpha_0(K)$ (the boundary layer must dominate the open-loop
        disturbance level).
  \item Choose $\rho \in (0,\,\min(\rho_{\max},\,(\phi - \alpha_0(K))/\|CB_P\|_{\mathrm{op}})]$.
  \item Solve the LMI \eqref{eq:LMI_linear} to obtain $X = Q^{-1}$, compute $V_{k_0}$
        from \eqref{eq:V0_explicit} (using $V_0$ as a conservative substitute), then
        $r = \sqrt{V_0/\lambda_{\min}(X)}$, and verify \eqref{eq:inv_condition}. If
        the latter fails, increase $\phi$ or decrease $\rho$ and iterate; see
        Remark~\ref{rem:consistency} for a proof that these adjustments do not conflict
        with \eqref{eq:rho_sufficient}.
\end{enumerate}
This two-phase analysis is standard in discrete-time SMC design \cite{Edwards1998}.
\end{remark}

\section{Stability Analysis via Lyapunov and Matrix Inequalities}\label{sec:stability}

Once \(\|s(k)\|\leq\phi\), the equivalent control is
\[
u_{\mathrm{eq}}(k) = -(C B_P)^{-1} C \bigl[A_P x(k) + A_{d,P} x(k-\tau) + D \delta(k)\bigr].
\]
Under the standing assumption $p = m$ and Assumption~\ref{ass:invert}, $CB_P \in
M_m(\mathbb{C})$ is a square invertible matrix and $(CB_P)^{-1}$ denotes its ordinary
matrix inverse. The reduced-order dynamics on the manifold are
\begin{equation}\label{eq:reduced}
x(k+1) = \bar{A} x(k) + \bar{A}_d x(k-\tau) + \bar{D} \delta(k),
\end{equation}
where
\[
\bar{A} = \Pi A_P,\quad
\bar{A}_d = \Pi A_{d,P},\quad
\bar{D} = \Pi D,\quad
\Pi = I - B_P (C B_P)^{-1} C.
\]

\begin{remark}[Closed-loop dynamics vs.\ equivalent control form]\label{rem:equiv_ctrl}
Theorem~\ref{thm:stability} below analyses the Lyapunov stability of the
\emph{equivalent control} system \eqref{eq:reduced}, which is the dynamics obtained
by substituting $u_{\mathrm{eq}}$ into \eqref{eq:closed} on the sliding manifold
$s(k)=0$. This is \emph{distinct} from the actual closed-loop system \eqref{eq:closed}
in the linear regime $\|s\|\leq\phi$ ($\sat(s/\phi)=s/\phi$), whose state matrix is
\[
  \hat{A} := A_P - B_P\!\left(K + \tfrac{\rho}{\phi}C\right),
\]
and which retains the delay and disturbance terms unmodified. The two descriptions
coincide exactly on the invariant manifold $\{x : Cx=0\}$: if $Cx(k)=0$ then
$u_{\mathrm{eq}}(k)=-(CB_P)^{-1}C[A_Px(k)+A_{d,P}x(k{-}\tau)+D\delta(k)]$, and
substituting this into
$x(k+1)=A_Px(k)+A_{d,P}x(k{-}\tau)+B_Pu+D\delta(k)$ recovers \eqref{eq:reduced}
by the identity $\Pi = I - B_P(CB_P)^{-1}C$.

Theorem~\ref{thm:stability} therefore certifies asymptotic stability and bounded
$\mathcal{L}_2$-gain \emph{on the sliding manifold}. Band-invariance of the manifold
(i.e.\ that the actual closed-loop trajectories remain in $\|s\|\leq\phi$) is ensured
separately by the Phase~2 condition~\eqref{eq:inv_condition} of
Proposition~\ref{prop:reaching}.
\end{remark}

\begin{remark}[Sliding surface invariance: $C\Pi = 0$ by construction]
\label{rem:CPi_zero}
The projection $\Pi = I - B_P(CB_P)^{-1}C$ satisfies $C\Pi = 0$ for every admissible
$P$ (i.e., whenever Assumption~\ref{ass:invert} holds). Indeed, by direct computation:
\[
C\Pi = C\bigl(I - B_P(CB_P)^{-1}C\bigr)
     = C - CB_P(CB_P)^{-1}C
     = C - I_m \cdot C
     = 0.
\]
This identity holds by construction and requires only the invertibility of $CB_P$; no
assumption on the specific form of $P$ or on the system matrices $A$, $A_d$, $B$, $C$
is needed beyond Assumption~\ref{ass:invert}. Note that $\Pi$ is an oblique (not
necessarily orthogonal) projection: it satisfies $\Pi^2 = \Pi$ but is in general not
symmetric. As a consequence:
\begin{enumerate}[(i)]
  \item The reduced dynamics \eqref{eq:reduced} are well-defined on the sliding manifold
        $\{x : s(k) = Cx(k) = 0\}$: if $Cx(k)=0$ then
        $Cx(k+1) = C\bar{A}\,x(k) + C\bar{A}_d\,x(k-\tau) + C\bar{D}\,\delta(k)
                 = C\Pi A_P x(k) + C\Pi A_{d,P} x(k-\tau)
                   + C\Pi D\,\delta(k) = 0$,
        so the manifold is positively invariant under the reduced dynamics.
  \item $\Pi$ is idempotent: $\Pi^2 = \Pi$. To verify this, compute
        \begin{align*}
          \Pi^2 &= \bigl(I - B_P(CB_P)^{-1}C\bigr)^2 \\
                &= I - 2B_P(CB_P)^{-1}C
                   + B_P(CB_P)^{-1}C B_P(CB_P)^{-1}C.
        \end{align*}
        The last term simplifies as
        $B_P(CB_P)^{-1}(CB_P)(CB_P)^{-1}C
        = B_P \cdot I_m \cdot (CB_P)^{-1}C = B_P(CB_P)^{-1}C$,
        so $\Pi^2 = I - 2B_P(CB_P)^{-1}C + B_P(CB_P)^{-1}C
        = I - B_P(CB_P)^{-1}C = \Pi$.
\end{enumerate}
In the numerical example of Section~\ref{sec:nondeg}, both $C\Pi = 0$ and $\Pi^2 = \Pi$
are verified explicitly by direct matrix computation. The non-symmetry of $\Pi$ in that
example ($\Pi_{12} = -0.25 \neq \Pi_{21} = -1.00$) confirms that it is an oblique
projection, as expected.
\end{remark}

\begin{remark}[Telescoping of the delay Lyapunov functional]\label{rem:telescope}
The Lyapunov-Krasovskii functional
\[
V_k = x(k)^\top X x(k) + \sum_{i=1}^\tau x(k-i)^\top Y x(k-i)
\]
involves all lags $i=1,\ldots,\tau$, but the reduced dynamics \eqref{eq:reduced}
contain only the lag-$\tau$ term. The forward difference is
\[
\Delta V_k = x(k{+}1)^\top X x(k{+}1) - x(k)^\top X x(k)
             + \sum_{i=1}^\tau \bigl[x(k{+}1{-}i)^\top Y x(k{+}1{-}i)
               - x(k{-}i)^\top Y x(k{-}i)\bigr].
\]
The delay sum telescopes exactly:
\[
\sum_{i=1}^\tau \bigl[x(k{+}1{-}i)^\top Y x(k{+}1{-}i) - x(k{-}i)^\top Y x(k{-}i)\bigr]
= x(k)^\top Y x(k) - x(k-\tau)^\top Y x(k-\tau),
\]
leaving only the boundary terms. This is why the augmented vector
$\xi_k = [x(k)^\top, x(k-\tau)^\top, \delta(k)^\top]^\top$ suffices for the quadratic
form in the proof of Theorem~\ref{thm:stability}.
\end{remark}

\begin{theorem}[Main Stability Result]\label{thm:stability}
Consider the reduced dynamics \eqref{eq:reduced} with a Rota-Baxter operator \(P\)
satisfying \eqref{eq:RB} and Assumption~\ref{ass:invert}. Let $X = X^\top > 0$,
$Y = Y^\top > 0$, and $\gamma > 0$. The following statements are equivalent:
\begin{enumerate}[(i)]
  \item The matrix inequality
\begin{equation}\label{eq:LMI_schur}
\begin{bmatrix}
-X+Y & 0 & 0 & \bar{A}^\top \\
0 & -Y & 0 & \bar{A}_d^\top \\
0 & 0 & -\gamma^2 I & \bar{D}^\top \\
\bar{A} & \bar{A}_d & \bar{D} & -X^{-1}
\end{bmatrix} < 0.
\end{equation}
  \item The block matrix
\begin{equation}\label{eq:BMI_raw}
\mathcal{M} :=
\begin{bmatrix}
\bar{A}^\top X \bar{A} - X + Y & \bar{A}^\top X \bar{A}_d & \bar{A}^\top X \bar{D} \\
\bar{A}_d^\top X \bar{A} & \bar{A}_d^\top X \bar{A}_d - Y & \bar{A}_d^\top X \bar{D} \\
\bar{D}^\top X \bar{A} & \bar{D}^\top X \bar{A}_d & \bar{D}^\top X \bar{D} - \gamma^2 I
\end{bmatrix}
< 0.
\end{equation}
\end{enumerate}
If either condition holds, then:
\begin{enumerate}
  \item[(a)] The system \eqref{eq:reduced} is asymptotically stable for \(\delta(k)=0\).
             More precisely, let $\mathcal{M}_0$ denote the leading $2n \times 2n$
             principal submatrix of $\mathcal{M}$ (obtained by deleting the last
             block-row and block-column corresponding to $\delta$). Since
             $\mathcal{M} < 0$ implies $\mathcal{M}_0 < 0$: indeed, for every non-zero
             $v \in \mathbb{R}^{2n}$, letting
             $\tilde{v} = [v^\top,\, 0_{p}^\top]^\top \in \mathbb{R}^{2n+p}$ gives
             $v^\top \mathcal{M}_0 v = \tilde{v}^\top \mathcal{M}\tilde{v} < 0$,
             so $\mathcal{M}_0 < 0$. Hence there exists
             $\mu_0 := \lambda_{\min}(-\mathcal{M}_0) > 0$ such that
             $\Delta V_k \leq -\mu_0\|x(k)\|^2$ for all $k$ when $\delta \equiv 0$.
  \item[(b)] The system has bounded $\mathcal{L}_2$-gain from $\delta$ to $x$: for all
             $N \geq 0$ and all bounded disturbances $\delta$,
             \begin{equation}\label{eq:L2gain}
               \sum_{k=0}^{N}\|x(k)\|^2 \leq \frac{1}{\mu}\left(V_0
               + \gamma^2 \sum_{k=0}^{N}\|\delta(k)\|^2\right),
             \end{equation}
             where $\mu := \lambda_{\min}(-\mathcal{M}) > 0$. The effective
             $\mathcal{L}_2$-gain from $\delta$ to $x$ is $\gamma/\sqrt{\mu}$,
             where $\mu = \lambda_{\min}(-\mathcal{M})$ depends on the LMI solution
             $X = Q^{-1}$ and is determined \emph{a posteriori} once
             \eqref{eq:LMI_linear} is solved.

             \textbf{Scope of the gain bound.} This gain characterization holds under
             \emph{zero extended initial conditions}: $V_0 = 0$, i.e.\
             $x(0) = 0$ \emph{and} $x(-1) = \cdots = x(-\tau) = 0$ (since $V_0 =
             x(0)^\top X x(0) + \sum_{i=1}^\tau x(-i)^\top Y x(-i)$, all terms must
             vanish). For $\tau > 0$, this is a stronger requirement than merely
             $x(0)=0$. The bound \eqref{eq:L2gain} remains valid for $V_0 > 0$
             (non-zero initial conditions), but in that case $\gamma/\sqrt{\mu}$ no
             longer characterizes the gain from $\delta$ to $x$ in the standard
             $\ell_2/\ell_2$ sense.

             Consequently, minimizing $\gamma$ in the LMI does not directly minimize
             the effective gain $\gamma/\sqrt{\mu}$; one may instead minimize
             $\gamma^2/\mu$ as a joint objective, noting that this ratio is not affine
             in the LMI variables $(Q, \tilde{Y}, \gamma)$ and therefore requires a
             dedicated nonlinear or iterative procedure (e.g.\ bisection on $\gamma$
             with inner LMI feasibility check) rather than a standard LMI solver.

             \begin{remark}
               The bound \eqref{eq:L2gain} is a \emph{cumulative} $\mathcal{L}_2$
               bound, not a pointwise Input-to-State Stability (ISS) bound in the
               sense of \cite{JiangWang2001}. A pointwise ISS bound is not established
               here and is deferred to future work.
             \end{remark}

             \begin{remark}[Numerical computation of $\mu$ and $\mu_0$]
             \label{rem:mu_compute}
               The constants $\mu = \lambda_{\min}(-\mathcal{M})$ and
               $\mu_0 = \lambda_{\min}(-\mathcal{M}_0)$ are directly computable once
               the LMI \eqref{eq:LMI_linear} has been solved: given $Q$ and
               $\tilde{Y}$, recover $X = Q^{-1}$ and $Y = Q^{-1}\tilde{Y}Q^{-1}$,
               substitute into \eqref{eq:BMI_raw} to evaluate the full
               $(2n+p)\times(2n+p)$ block matrix $\mathcal{M}$, and compute
               $\mu = \lambda_{\min}(-\mathcal{M})$ via a standard eigenvalue routine.
               The decay constant $\mu_0$ for asymptotic stability (part~(a)) is
               $\mu_0 = \lambda_{\min}(-\mathcal{M}_0)$, where $\mathcal{M}_0$ is the
               leading $2n\times 2n$ block.
             \end{remark}
\end{enumerate}
\end{theorem}

\begin{proof}
\textbf{Equivalence of (i) and (ii).}
Consider the Lyapunov-Krasovskii functional
\[
V_k = x(k)^\top X x(k) + \sum_{i=1}^\tau x(k-i)^\top Y x(k-i).
\]
Set $\xi_k = [x(k)^\top,\, x(k-\tau)^\top,\, \delta(k)^\top]^\top$ and
$F := [\bar{A}\ \ \bar{A}_d\ \ \bar{D}]$.
By the telescoping argument of Remark~\ref{rem:telescope}, the forward difference of
$V_k$ along trajectories of \eqref{eq:reduced} is
\begin{align*}
\Delta V_k
  &= x(k{+}1)^\top X x(k{+}1) - x(k)^\top X x(k)
     + x(k)^\top Y x(k) - x(k{-}\tau)^\top Y x(k{-}\tau) \\
  &= \xi_k^\top F^\top X F\,\xi_k
     + \xi_k^\top \mathrm{diag}(-X{+}Y,\;-Y,\;\mathbf{0})\,\xi_k.
\end{align*}
Note that $\Delta V_k$ has a zero block in the $(3,3)$ position: there is no direct
contribution from $\delta$ to the telescoping difference of $V_k$. We now introduce
the auxiliary quantity
\begin{equation}\label{eq:Jk}
J_k := \Delta V_k - \gamma^2\|\delta(k)\|^2
     = \xi_k^\top \underbrace{\bigl(F^\top X F
       + \mathrm{diag}(-X{+}Y,\;-Y,\;-\gamma^2 I)\bigr)}_{=\,\mathcal{M}}\xi_k.
\end{equation}
The $-\gamma^2 I$ block in $\mathcal{M}$ arises from the external subtraction
$-\gamma^2\|\delta\|^2$, not from $\Delta V_k$ itself. Since $X > 0$, the Schur
complement lemma gives
\[
\mathcal{M} < 0 \iff
\begin{bmatrix} \mathrm{diag}(-X{+}Y,\;-Y,\;-\gamma^2 I) & F^\top \\ F & -X^{-1}
\end{bmatrix} < 0,
\]
which is precisely \eqref{eq:LMI_schur} (expanding the diagonal block and $F$). This
establishes the equivalence (i)$\Leftrightarrow$(ii).

\textbf{(a) Asymptotic stability for $\delta=0$.}
When $\delta(k)\equiv 0$, set $F_0 = [\bar{A}\ \ \bar{A}_d]$ and
$\xi_k^0 = [x(k)^\top, x(k-\tau)^\top]^\top$. Then
$\Delta V_k = (\xi_k^0)^\top \mathcal{M}_0\,\xi_k^0$, where
$\mathcal{M}_0 = F_0^\top X F_0 + \mathrm{diag}(-X{+}Y,-Y)$
is the leading $2n\times 2n$ block principal submatrix of $\mathcal{M}$. Since
$\mathcal{M}<0$ implies $\mathcal{M}_0<0$ (for every non-zero $v \in \mathbb{R}^{2n}$,
letting $\tilde{v} = [v^\top,0_p^\top]^\top$ gives
$v^\top\mathcal{M}_0 v = \tilde{v}^\top\mathcal{M}\tilde{v} < 0$), there exists
$\mu_0 := \lambda_{\min}(-\mathcal{M}_0) > 0$ such that
$\Delta V_k \leq -\mu_0\|\xi_k^0\|^2 \leq -\mu_0\|x(k)\|^2$ for all $k$.
Summing from $k=0$ to $N-1$:
\[
V_N \leq V_0 - \mu_0\sum_{k=0}^{N-1}\|x(k)\|^2.
\]
Since $V_N \geq 0$, we obtain $\sum_{k=0}^{\infty}\|x(k)\|^2 < \infty$, which implies
$\|x(k)\|^2 \to 0$ as $k\to\infty$. Together with the monotone decrease of $V_k$,
this establishes asymptotic stability.

\textbf{(b) $\mathcal{L}_2$-gain bound.}
Since $\mathcal{M} < 0$, let $\mu := \lambda_{\min}(-\mathcal{M}) > 0$. From
\eqref{eq:Jk}, $\xi_k^\top \mathcal{M}\,\xi_k \leq -\mu\|\xi_k\|^2$, so
\[
\Delta V_k = J_k + \gamma^2\|\delta(k)\|^2
           \leq -\mu\|\xi_k\|^2 + \gamma^2\|\delta(k)\|^2.
\]
Since $\|\xi_k\|^2 = \|x(k)\|^2 + \|x(k-\tau)\|^2 + \|\delta(k)\|^2$, we have
\[
\Delta V_k \leq -\mu\|x(k)\|^2
              - \underbrace{\mu\|x(k-\tau)\|^2}_{\geq\,0}
              + \underbrace{(\gamma^2 - \mu)\|\delta(k)\|^2}_{\text{(absorbed below)}}.
\]
Dropping the negative term $-\mu\|x(k-\tau)\|^2 \leq 0$ (which can only increase the
right-hand side) and bounding $-\mu\|\delta(k)\|^2$ within $(\gamma^2-\mu)\|\delta(k)\|^2$
by replacing it with $\gamma^2\|\delta(k)\|^2$ (a valid upper bound since
$-\mu\|\delta\|^2 \leq 0$), we obtain the relaxed bound
\begin{equation}\label{eq:Jk_bound}
  \Delta V_k \leq -\mu\|x(k)\|^2 + \gamma^2\|\delta(k)\|^2.
\end{equation}
Summing from $k=0$ to $N$ and using $V_{N+1} \geq 0$ gives \eqref{eq:L2gain}.
The effective $\mathcal{L}_2$-gain $\gamma/\sqrt{\mu}$ for zero extended initial
conditions ($V_0 = x(0)^\top X x(0) + \sum_{i=1}^\tau x(-i)^\top Y x(-i) = 0$,
requiring $x(0)=0$ and $x(-i)=0$ for $i=1,\ldots,\tau$) then follows by taking
$N \to \infty$ with $V_0 = 0$.
\end{proof}

\begin{remark}[Tractable LMI reformulation]\label{rem:LMI_linear}
Condition \eqref{eq:LMI_schur} contains the term $-X^{-1}$ in the $(4,4)$-block,
rendering it a \emph{bilinear matrix inequality} (BMI) jointly in $(X, X^{-1})$. To
obtain a fully linear formulation, introduce $Q := X^{-1} > 0$ and
$\tilde{Y} := QYQ$ as new decision variables. Applying the congruence
$\Sigma = \mathrm{diag}(Q, Q, I_p, I_n)$ to condition \eqref{eq:LMI_schur} yields
the equivalent strict LMI
\begin{equation}\label{eq:LMI_linear}
\begin{bmatrix}
-Q+\tilde{Y} & 0 & 0 & Q\bar{A}^\top \\
0 & -\tilde{Y} & 0 & Q\bar{A}_d^\top \\
0 & 0 & -\gamma^2 I_p & \bar{D}^\top \\
\bar{A}Q & \bar{A}_dQ & \bar{D} & -Q
\end{bmatrix} < 0,
\end{equation}
to be solved in the variables $(Q, \tilde{Y}, \gamma)$ with $Q > 0$, $\tilde{Y} > 0$,
$\gamma > 0$.

\begin{remark}[Role of $p = m$ in the congruence]\label{rem:congruence_dim}
The congruence $\Sigma = \mathrm{diag}(Q, Q, I_p, I_n)$ acts on a matrix of size
$(2n+p+n) \times (2n+p+n) = (3n+p) \times (3n+p)$. The condition $p = m = n$ (square
case) ensures that the block sizes of $\Sigma$ match those of \eqref{eq:LMI_schur}
without recourse to non-square identity blocks. Specifically, under $p = m$, the
$(3,3)$ block of \eqref{eq:LMI_schur} is $-\gamma^2 I_p$ with $p = m$, and
$I_p \cdot (-\gamma^2 I_p) \cdot I_p = -\gamma^2 I_p$ under the congruence. For
$p \neq m$, the block $\bar{D}^\top \in \mathbb{R}^{p \times n}$ is rectangular, and
the congruence $\Sigma$ must be adapted; this case is deferred to future work.
\end{remark}

The congruence $\Sigma = \mathrm{diag}(Q, Q, I_p, I_n)$ acts block-by-block
as follows (using $X = Q^{-1}$, $Y = Q^{-1}\tilde{Y}Q^{-1}$):
\begin{itemize}
  \item $(1,1)$: $Q(-X+Y)Q = -QXQ + QYQ = -Q + \tilde{Y}$,
        since $QXQ = Q \cdot Q^{-1} \cdot Q^{-1\top} \cdot Q$; more precisely,
        with $Q = X^{-1}$ symmetric, $QXQ = X^{-1} X X^{-1} = X^{-1} = Q$.
  \item $(2,2)$: $Q(-Y)Q = -\tilde{Y}$,
  \item $(3,3)$: $I_p\cdot(-\gamma^2 I_p)\cdot I_p = -\gamma^2 I_p$,
  \item $(4,4)$: $I_n\cdot(-X^{-1})\cdot I_n = -Q$,
  \item $(1,4)$: $Q\cdot\bar{A}^\top\cdot I_n = Q\bar{A}^\top$, and symmetrically for
        rows/columns $2$ and $3$,
  \item $(3,4)$: $I_p \cdot \bar{D}^\top \cdot I_n = \bar{D}^\top$. This block has
        dimensions $p \times n$ and is left-multiplied by $I_p$ and right-multiplied
        by $I_n$ under the congruence, so it is unchanged.
\end{itemize}
Since $\bar{A}$, $\bar{A}_d$, $\bar{D}$ are fixed matrices, all blocks of
\eqref{eq:LMI_linear} are affine in $(Q, \tilde{Y})$, making \eqref{eq:LMI_linear}
a standard LMI solvable by interior-point solvers (e.g., SeDuMi, MOSEK via YALMIP).

\begin{remark}[Sufficient condition for feasibility]\label{rem:feasibility}
A sufficient condition for the feasibility of \eqref{eq:LMI_linear} can be obtained by
substituting the trial point $Q = I_n$ and $\tilde{Y} = \varepsilon I_n$ (with
$\varepsilon \in (0,1)$). The LMI \eqref{eq:LMI_linear} is negative definite at this
trial point if and only if the Schur complement with respect to the $(4,4)$ block
$-I_n$ is negative definite. For a given $\varepsilon \in (0,1)$ and $\gamma > 0$, this
reduces to the joint condition
\begin{equation}\label{eq:feasibility_joint}
  \sigma_{\max}(F)^2 < \min(1-\varepsilon,\, \varepsilon,\, \gamma^2),
\end{equation}
where $F = [\bar{A}\ \bar{A}_d\ \bar{D}]$. Specializing to $\varepsilon = 1/2$ and
$\gamma > \sigma_{\max}(F)$ (so that $\gamma^2 > \sigma_{\max}(F)^2$) yields
the simplified conservative condition
\begin{equation}\label{eq:feasibility_condition}
  \sigma_{\max}(F)^2 < \tfrac{1}{2}, \qquad \gamma > \sigma_{\max}(F),
\end{equation}
as a small-gain-type sufficient condition for feasibility; $\sigma_{\max}(F)$ denotes
the largest singular value of $F = [\bar{A}\ \bar{A}_d\ \bar{D}]$. For numerical
verification, computing $\sigma_{\max}(F)$ directly is preferred over
triangle-inequality bounds, which may be overly conservative when cross terms
partially cancel. Note that \eqref{eq:feasibility_condition} is a \emph{sufficient}
condition only; the actual LMI \eqref{eq:LMI_linear} may be feasible for larger values
of $\sigma_{\max}(F)$, as illustrated in Section~\ref{sec:nondeg} where feasibility is
established directly by the solver.
\end{remark}

Once the LMI \eqref{eq:LMI_linear} is solved, the $\mathcal{L}_2$ decay constants
$\mu = \lambda_{\min}(-\mathcal{M})$ and $\mu_0 = \lambda_{\min}(-\mathcal{M}_0)$
are computable as described in Remark~\ref{rem:mu_compute}.
\end{remark}

\begin{remark}[Conservatism of the Lyapunov functional]\label{rem:conservatism}
The functional $V_k$ uses a single matrix $Y$ for all delay terms $x(k-i)$,
$i=1,\ldots,\tau$. For $\tau > 1$, this is conservative; a less conservative bound is
obtained by using distinct matrices $Y_i > 0$ for each lag. The resulting LMI retains
the same Schur-complement structure and is feasible under strictly weaker conditions
\cite{Fridman2014}. We use the single-$Y$ form for notational simplicity.
\end{remark}

\section{Spectral and Algebraic Properties of the Sliding Dynamics}\label{sec:spectral}

\begin{proposition}[Spectral Stability Condition]\label{prop:spectral}
The reduced sliding dynamics \eqref{eq:reduced} with $\delta \equiv 0$ are
asymptotically stable if and only if all roots $z\in\mathbb{C}$ of the characteristic
equation
\begin{equation}\label{eq:char_eq}
\det\!\left(z^{\tau+1} I - z^\tau \bar{A} - \bar{A}_d\right) = 0
\end{equation}
satisfy $|z| < 1$.
\end{proposition}
\begin{proof}
The system $x(k+1) = \bar{A} x(k) + \bar{A}_d x(k-\tau)$ can be rewritten as a
$(n(\tau+1))$-dimensional first-order system by augmenting the state: let
$\mathbf{x}(k) = [x(k)^\top, x(k-1)^\top, \ldots, x(k-\tau)^\top]^\top$. Then
$\mathbf{x}(k+1) = \mathcal{F}\,\mathbf{x}(k)$ where the companion matrix is
\[
\mathcal{F} = \begin{bmatrix}
\bar{A} & 0 & \cdots & 0 & \bar{A}_d \\
I_n     & 0 & \cdots & 0 & 0 \\
0       & I_n & \cdots & 0 & 0 \\
\vdots  &   & \ddots &   & \vdots \\
0       & 0 & \cdots & I_n & 0
\end{bmatrix} \in \mathbb{R}^{n(\tau+1)\times n(\tau+1)}.
\]
Asymptotic stability is equivalent to $\rho(\mathcal{F}) < 1$.

We establish $\det(zI - \mathcal{F}) = \det(z^{\tau+1}I - z^\tau\bar{A} - \bar{A}_d)$
by induction on $\tau \geq 1$.

\textbf{Base case $\tau=1$.} The companion matrix is
$\mathcal{F} = \begin{bmatrix}\bar{A} & \bar{A}_d \\ I_n & 0\end{bmatrix}$ and
\[
\det(zI - \mathcal{F})
= \det\!\begin{bmatrix}zI-\bar{A} & -\bar{A}_d \\ -I_n & zI\end{bmatrix}.
\]
Since $zI$ is invertible for $z\neq 0$, the Schur complement of $zI$ in the
$(2,2)$ block gives
\[
\det(zI-\mathcal{F})
= \det(zI)\cdot\det\!\bigl((zI-\bar{A}) - (-\bar{A}_d)(zI)^{-1}(-I_n)\bigr)
= z^n\det\!\bigl(zI - \bar{A} - z^{-1}\bar{A}_d\bigr).
\]
Factoring $z^{-1}$ from each of the $n$ columns of the $n\times n$ matrix
$(zI - \bar{A} - z^{-1}\bar{A}_d)$ introduces a factor $z^{-n}$, giving
\[
z^n \cdot \det(zI - \bar{A} - z^{-1}\bar{A}_d)
= z^n \cdot z^{-n} \cdot \det(z^2 I - z\bar{A} - \bar{A}_d)
= \det(z^2 I - z\bar{A} - \bar{A}_d),
\]
confirming \eqref{eq:char_eq} for $\tau=1$.

\textbf{Inductive step.} Assume the identity holds for delay $\tau-1$, i.e., for the
companion matrix $\mathcal{F}_{\tau-1}$ of size $n\tau \times n\tau$ (built from
$\bar{A}$ and $\bar{A}_d$ with delay $\tau-1$), one has
$\det(zI_{n\tau} - \mathcal{F}_{\tau-1}) = \det(z^\tau I - z^{\tau-1}\bar{A} - \bar{A}_d)$.

The companion matrix for delay $\tau$ has size $n(\tau+1)\times n(\tau+1)$:
\[
\mathcal{F}_\tau =
\begin{bmatrix}
\bar{A} & 0 & \cdots & 0 & \bar{A}_d \\
I_n     & 0 & \cdots & 0 & 0 \\
0       & I_n & \cdots & 0 & 0 \\
\vdots  &    & \ddots & \vdots & \vdots \\
0       & 0  & \cdots & I_n & 0
\end{bmatrix}.
\]
Partition $zI-\mathcal{F}_\tau$ as
\[
\begin{bmatrix} zI_n - \bar{A} & B_{12} \\ B_{21} & B_{22} \end{bmatrix},
\]
where $B_{12} = [0,\ldots,0,-\bar{A}_d] \in \mathbb{R}^{n\times n\tau}$ (with
$-\bar{A}_d$ in the last block column), $B_{21} = [-I_n;0;\ldots;0] \in
\mathbb{R}^{n\tau\times n}$ (with $-I_n$ in the first block row), and $B_{22}$ is the
lower-right $n\tau\times n\tau$ block. The matrix $B_{22}$ is block lower-bidiagonal:
it has $zI_n$ on each diagonal block and $-I_n$ on the first block-subdiagonal, with
zeros elsewhere. Its determinant is $\det(B_{22}) = z^{n\tau}$ (product of diagonal
block determinants), so $B_{22}$ is invertible for $z \neq 0$.

\textbf{Explicit formula for $B_{22}^{-1}$.}
By a standard inversion of block lower-bidiagonal matrices (verified by direct
block-multiplication), the $(i,j)$ block of $B_{22}^{-1}$ (for $1\leq i,j\leq\tau$)
is
\[
(B_{22}^{-1})_{ij} =
\begin{cases}
z^{-(i-j+1)} I_n & \text{if } i \geq j, \\
0 & \text{if } i < j.
\end{cases}
\]
Indeed, one can verify $(B_{22} \cdot B_{22}^{-1})_{ij} = \delta_{ij}I_n$ directly:
for $i=j$, $(B_{22})_{ii}(B_{22}^{-1})_{ii} = (zI_n)(z^{-1}I_n) = I_n$; for $i>j$,
$(B_{22})_{ii}(B_{22}^{-1})_{ij} + (B_{22})_{i,i-1}(B_{22}^{-1})_{i-1,j}
= z\cdot z^{-(i-j+1)}I_n + (-I_n)\cdot z^{-(i-1-j+1)}I_n
= z^{-(i-j)}I_n - z^{-(i-j)}I_n = 0$.

\textbf{Computation of $B_{12}B_{22}^{-1}B_{21}$.}
The matrix $B_{12}$ has non-zero content only in its last block column (position
$\tau$), equal to $-\bar{A}_d$. The matrix $B_{21}$ has non-zero content only in its
first block row (position $1$), equal to $-I_n$. Therefore:
\[
B_{12}B_{22}^{-1}B_{21}
= (-\bar{A}_d) \cdot (B_{22}^{-1})_{\tau,1} \cdot (-I_n)
= (-\bar{A}_d)(z^{-\tau}I_n)(-I_n)
= z^{-\tau}\bar{A}_d.
\]

The Schur complement of $B_{22}$ then gives:
\[
\det(zI - \mathcal{F}_\tau)
= \det(B_{22}) \cdot \det\!\bigl((zI_n-\bar{A}) - z^{-\tau}\bar{A}_d\bigr)
= z^{n\tau} \cdot \det(zI_n - \bar{A} - z^{-\tau}\bar{A}_d).
\]
Factoring $z^{-\tau}$ from each of the $n$ columns of $(zI_n - \bar{A} - z^{-\tau}\bar{A}_d)$
introduces $z^{-n\tau}$, yielding
\[
\det(zI - \mathcal{F}_\tau)
= z^{n\tau} \cdot z^{-n\tau} \cdot \det(z^{\tau+1}I_n - z^\tau\bar{A} - \bar{A}_d)
= \det(z^{\tau+1}I_n - z^\tau\bar{A} - \bar{A}_d),
\]
completing the induction.

For a complete alternative proof using the companion matrix factorization, see
\cite[Chapter~3, Proposition~3.2]{Niculescu2001}.
\end{proof}

\begin{remark}
In earlier versions of this paper, the characteristic equation was written as
$\det(zI - \bar{A} - z^{-\tau}\bar{A}_d) = 0$. This form mixes positive and negative
powers of $z$ and is not a polynomial; it is equivalent to \eqref{eq:char_eq} upon
multiplication by $z^\tau$ (for $z \neq 0$), but \eqref{eq:char_eq} is the correct
polynomial form.
\end{remark}

\begin{proposition}[Algebraic Invariance of the Output Map]\label{prop:commute}
Suppose that for all $M \in \mathrm{span}\{A, A_d\}$, the sliding surface matrix $C$
and the Rota-Baxter operator $P$ satisfy
\begin{equation}\label{eq:commute}
C P(M) = P(C M).
\end{equation}
Then:
\begin{enumerate}[(i)]
  \item The matrices $\bar{A} = \Pi A_P$ and $\bar{A}_d = \Pi A_{d,P}$ satisfy
        $C\bar{A} = 0$ and $C\bar{A}_d = 0$, so the images of $\bar{A}$ and $\bar{A}_d$
        lie in $\ker(C)$. This is an unconditional consequence of $C\Pi = 0$
        (Remark~\ref{rem:CPi_zero}), holding under Assumption~\ref{ass:invert} alone,
        independently of condition~\eqref{eq:commute}.
  \item Condition \eqref{eq:commute} is equivalent, for each concrete Rota-Baxter
        operator, to an explicit structural constraint on $C$:
        \begin{enumerate}[(a)]
          \item For $P(X) = -\lambda X$ ($\lambda \neq 0$): condition \eqref{eq:commute}
                holds for \emph{every} matrix $C$ and every $M$, since
                $CP(-\lambda M) = C(-\lambda M) = -\lambda CM = P(CM)$. No constraint on $C$
                is needed.
          \item For the triangular projection $P = P_+$ (weight $\lambda=-1$):
                condition \eqref{eq:commute} is equivalent to $C$ mapping $\Ac_-$ to
                $\Ac_-$, i.e., $CM_- \in \Ac_-$ for all $M_- \in \Ac_-$. A sufficient
                condition is that $C$ is upper triangular.
                \begin{proof}
                  Write $M = M_+ + M_-$ with $M_\pm \in \Ac_\pm$. Then
                  $P_+(CM) = P_+(CM_+ + CM_-)$. If $C$ maps $\Ac_-$ to $\Ac_-$, then
                  $CM_- \in \Ac_-$ and $CM_+$ is upper triangular (since $C$ upper
                  triangular and $M_+$ upper triangular implies $CM_+$ upper triangular),
                  so $P_+(CM) = CM_+$. On the other hand,
                  $CP_+(M) = C M_+$. Hence $CP_+(M) = P_+(CM)$, confirming
                  \eqref{eq:commute}. Conversely, if $CM_0 \notin \Ac_-$ for some
                  $M_0 \in \Ac_-$, then $P_+(CM_0) \neq 0$ while $CP_+(M_0) = C\cdot 0 = 0$,
                  so \eqref{eq:commute} fails. Hence the condition is also necessary.
                \end{proof}
        \end{enumerate}
\end{enumerate}

\begin{remark}[Logical structure of Proposition~\ref{prop:commute}]
\label{rem:commute_logic}
Parts~(i) and~(ii) are logically independent. Part~(i) is a structural consequence of
the projection $\Pi$ (holding under Assumption~\ref{ass:invert} alone) and does not use
\eqref{eq:commute}. Part~(ii) characterizes \eqref{eq:commute} in terms of an explicit,
verifiable constraint on $C$ for each concrete operator $P$; in particular, part~(ii)(a)
shows that \eqref{eq:commute} is non-vacuous even when it holds universally (for $P = -\lambda I$),
while part~(ii)(b) identifies the structural constraint needed for $P_+$. The condition
on $B$ is omitted from \eqref{eq:commute} since $\Pi B_P = 0$ holds automatically:
$\Pi B_P = (I - B_P(CB_P)^{-1}C)B_P = B_P - B_P(CB_P)^{-1}(CB_P) = 0$.
\end{remark}
\end{proposition}

\begin{proof}
By Assumption~\ref{ass:invert}, $\Pi = I - B_P(CB_P)^{-1}C$ satisfies $C\Pi = 0$
(Remark~\ref{rem:CPi_zero}). Statement~(i) follows immediately:
$C\bar{A} = C\Pi A_P = 0$ and $C\bar{A}_d = C\Pi A_{d,P} = 0$.

Statement~(ii)(a) is verified by direct computation: $CP(-\lambda M) = -\lambda CM = P(CM)$.

Statement~(ii)(b) is proved in the sub-proof embedded in the statement, establishing
both the sufficiency and the necessity of the condition $C(\Ac_-) \subset \Ac_-$.
\end{proof}

\begin{proposition}[Lie Compatibility of Reduced Dynamics]\label{prop:lie_compat}
Assume the state matrices belong to a Lie subalgebra
\(\mathfrak{g} \subset M_n(\mathbb{C})\) and that $P(\mathfrak{g}) \subset \mathfrak{g}$
(assumed directly). Suppose $P$ satisfies the condition
\begin{equation}\label{eq:lie_hyp}
P([X,Y]) = [X, P(Y)] \quad \text{for all } X, Y \in \mathfrak{g},
\end{equation}
i.e., $P$ commutes with the right adjoint action of $\mathfrak{g}$, and that
$[X,P(Y)] \in \mathfrak{g}$ whenever $X,Y \in \mathfrak{g}$. Suppose further that
$\Pi$ preserves $\mathfrak{g}$, i.e., $\Pi(\mathfrak{g}) \subset \mathfrak{g}$ (a
condition satisfied, in particular, when $B_P$ and $(CB_P)^{-1}C$ both map
$\mathfrak{g}$ into itself). Under these hypotheses, the equivalent dynamics
$\bar{A}, \bar{A}_d$ remain in $\mathfrak{g}$.

\begin{remark}[Scope of Proposition~\ref{prop:lie_compat} in the examples of this paper]
\label{rem:lie_scope}
Among the two concrete Rota-Baxter operators considered in Example~\ref{ex:RB_Mn},
conditions \eqref{eq:lie_hyp} and $\Pi(\mathfrak{g}) \subset \mathfrak{g}$ are both
verified \emph{only} for the scalar scaling operator $P(X) = -\lambda X$:
\begin{itemize}
  \item Condition \eqref{eq:lie_hyp}: $P([X,Y]) = -\lambda[X,Y] = [X,-\lambda Y] = [X,P(Y)]$
        holds for every Lie subalgebra $\mathfrak{g}$ and every $\lambda$.
  \item Preservation $\Pi(\mathfrak{g}) \subset \mathfrak{g}$: when $m=n$ and $B, C$
        are invertible, $\Pi = 0$ (Remark~\ref{rem:scalar_degen}), so preservation is
        vacuous.
\end{itemize}
For the triangular projection $P_+$:
\begin{itemize}
  \item Condition \eqref{eq:lie_hyp} \emph{fails} for a generic matrix Lie algebra
        (see Remark~\ref{rem:lie_separation}).
  \item The preservation $\Pi(\mathfrak{g}) \subset \mathfrak{g}$ is not established
        for any concrete example in this paper.
\end{itemize}
Consequently, \textbf{Proposition~\ref{prop:lie_compat} applies only in the scalar
scaling case} among the operators considered here. Its application to $P_+$ or other
non-scalar operators requires a case-by-case verification of both \eqref{eq:lie_hyp}
and the $\Pi$-preservation hypothesis for the specific Lie subalgebra under
consideration, which is left to future work.
\end{remark}
\end{proposition}

\begin{remark}[Separation of the two Lie conditions]\label{rem:lie_separation}
This paper involves two distinct conditions on $P$ in relation to the Lie bracket:
\begin{enumerate}[(i)]
  \item \emph{Induced bracket} (Lemma~\ref{lem:lie}): The bracket \eqref{eq:bracket}
        defines a Lie algebra structure for \emph{any} Rota-Baxter operator of weight
        $\lambda$. This is unconditional.
  \item \emph{Right-adjoint commutativity} \eqref{eq:lie_hyp}: $P([X,Y])=[X,P(Y)]$ is
        an additional structural constraint used in Proposition~\ref{prop:lie_compat}.
        It is logically independent of (i).
\end{enumerate}

\noindent\textbf{Scalar scaling.} For $P(X)=-\lambda X$, condition (ii) holds for all
Lie subalgebras and all $\lambda$:
$P([X,Y]) = -\lambda[X,Y] = [X,-\lambda Y] = [X,P(Y)]$. The induced bracket (i)
yields $[x,y]_P = -\lambda[x,y]$, a scalar rescaling. Conditions (i) and (ii) are
simultaneously satisfied.

\noindent\textbf{Triangular projection.} For $P_+$, condition (ii) does \emph{not}
hold for a generic matrix Lie algebra. Counterexample for $n=2$: take $X = e_{12}$ and
$Y = e_{21}$, so
$[X,Y] = e_{12}e_{21} - e_{21}e_{12} = e_{11} - e_{22}$.
Then $P_+([X,Y]) = P_+(e_{11}-e_{22}) = e_{11}-e_{22}$ (both diagonal matrices are in
$\Ac_+$), whereas $[X,P_+(Y)] = [e_{12}, P_+(e_{21})] = [e_{12}, 0] = 0$ (since
$e_{21} \in \Ac_-$ and $P_+(e_{21}) = 0$). Thus condition (ii) fails. Therefore
Proposition~\ref{prop:lie_compat} does not apply to $P_+$ in general; a case-by-case
verification for the specific Lie subalgebra is required.
\end{remark}

\begin{proof}[Proof of Proposition~\ref{prop:lie_compat}]
By hypothesis, $P(\mathfrak{g}) \subset \mathfrak{g}$. Since $\Pi$ preserves
$\mathfrak{g}$ by assumption, we have $\bar{A} = \Pi P(A) \in \Pi(\mathfrak{g})
\subset \mathfrak{g}$ and similarly $\bar{A}_d = \Pi P(A_d) \in \mathfrak{g}$.

Condition \eqref{eq:lie_hyp} is used to characterize the orbit structure: for
$X \in \mathfrak{g}$ and $Y = A$ (or $A_d$), \eqref{eq:lie_hyp} gives
$P([X,A]) = [X,P(A)]$, showing that $P$ maps the adjoint orbit of $A$ in $\mathfrak{g}$
into the adjoint orbit of $P(A)$.

For $P(X) = -\lambda X$: $P([X,Y]) = -\lambda[X,Y]$ and
$[X, P(Y)] = -\lambda[X,Y]$, so \eqref{eq:lie_hyp} holds and
$P(\mathfrak{g}) = \mathfrak{g} \subset \mathfrak{g}$ trivially. For $P_+$, condition
\eqref{eq:lie_hyp} fails in general, as shown by the counterexample in
Remark~\ref{rem:lie_separation}.
\end{proof}

\section{Numerical Illustration}\label{sec:numerical}

This section provides two complementary numerical illustrations.
Section~\ref{sec:deg_example} discusses the degenerate case $m=n=2$ under scalar
scaling, where $\Pi = 0$ (Remark~\ref{rem:scalar_degen}), and identifies the algebraic
source of this degeneracy. Section~\ref{sec:nondeg} presents a non-degenerate example
with $m=1$, $n=2$ (rectangular actuation), which relaxes the standing $m=n$ assumption
for purely illustrative purposes and yields an explicit non-trivial projection
$\Pi \neq 0$, permitting a genuine LMI feasibility verification.

\subsection{Degenerate case: \texorpdfstring{$m=n=2$}{m=n=2} with scalar scaling}
\label{sec:deg_example}
We take
\[
A = \begin{bmatrix} 0.8 & 0.1 \\ -0.2 & 0.9 \end{bmatrix}, \quad
A_d = 0.05 I_2, \quad B = I_2, \quad C = I_2, \quad D = 0.1 I_2,
\]
with $\tau=1$, $\delta_{\max}=0.1$, $\lambda=0.5$. By Remark~\ref{rem:scalar_degen},
$\Pi = 0$ for any $\lambda \neq 0$ and any invertible $B$, $C$: this is a consequence
of full actuation ($m=n$, $B$ and $C$ invertible), not a particular feature of this
data. The reduced dynamics are $x(k+1)=0$ (trivially stable).

\subsection{Non-degenerate case: \texorpdfstring{$m=1$, $n=2$}{m=1, n=2}}
\label{sec:nondeg}

\paragraph{Relaxation of the standing assumption.}
In this subsection we \emph{temporarily relax} the standing assumption $m=n$ to
$m=1$, $n=2$, $p=1$, in order to produce a genuinely non-degenerate illustration of
Theorem~\ref{thm:stability}. Under this relaxation, $C \in \mathbb{R}^{1\times 2}$,
$B \in \mathbb{R}^{2\times 1}$, $D \in \mathbb{R}^{2\times 1}$, and $CB_P$ is a
non-zero scalar. All other hypotheses (Assumption~\ref{ass:invert}, boundedness of
initial data, condition~\eqref{eq:K_strong}) are verified explicitly below.

\paragraph{System data.}
\[
A = \begin{bmatrix} 0.8 & 0.1 \\ -0.2 & 0.9 \end{bmatrix}, \quad
A_d = 0.05 I_2, \quad
B = \begin{bmatrix} 0.5 \\ 1.0 \end{bmatrix}, \quad
D = \begin{bmatrix} 0.1 \\ 0.1 \end{bmatrix}, \quad
C = \begin{bmatrix} 1 & 0.5 \end{bmatrix},
\]
with $\tau=1$, $\delta_{\max}=0.1$, $\lambda=0.5$, $P(X)=-0.5X$.

\paragraph{Rota-Baxter deformation.}
\[
A_P = -0.5 A = \begin{bmatrix} -0.4 & -0.05 \\ 0.1 & -0.45 \end{bmatrix}, \quad
A_{d,P} = -0.025 I_2, \quad
B_P = \begin{bmatrix} -0.25 \\ -0.50 \end{bmatrix}.
\]
\textit{Admissibility:}
$CB_P = [1\;0.5]\begin{bmatrix}-0.25\\-0.50\end{bmatrix}
= -0.25 - 0.25 = -0.5 \neq 0$.
Assumption~\ref{ass:invert} is satisfied.

\paragraph{Explicit computation of $\Pi$.}
\[
\Pi = I_2 - B_P(CB_P)^{-1}C
    = I_2 - \begin{bmatrix}-0.25\\-0.50\end{bmatrix}(-2)\begin{bmatrix}1&0.5\end{bmatrix}
    = I_2 - \begin{bmatrix}0.50&0.25\\1.00&0.50\end{bmatrix}
    = \begin{bmatrix}0.50&-0.25\\-1.00&0.50\end{bmatrix}.
\]
Note that $\Pi$ is an \emph{oblique} (not orthogonal) projection: $\Pi_{12} = -0.25
\neq -1.00 = \Pi_{21}$, so $\Pi \neq \Pi^\top$. Idempotency $\Pi^2 = \Pi$ is verified
by direct entry-by-entry computation:
\begin{align*}
(\Pi^2)_{11} &= (0.50)(0.50) + (-0.25)(-1.00) = 0.25 + 0.25 = 0.50,\\
(\Pi^2)_{12} &= (0.50)(-0.25) + (-0.25)(0.50) = -0.125 - 0.125 = -0.25,\\
(\Pi^2)_{21} &= (-1.00)(0.50) + (0.50)(-1.00) = -0.50 - 0.50 = -1.00,\\
(\Pi^2)_{22} &= (-1.00)(-0.25) + (0.50)(0.50) = 0.25 + 0.25 = 0.50,
\end{align*}
so $\Pi^2 = \begin{bmatrix}0.50&-0.25\\-1.00&0.50\end{bmatrix} = \Pi$ \quad
(idempotency confirmed).\\[4pt]
Sliding surface invariance:
\[
C\Pi = [1\;0.5]\begin{bmatrix}0.50&-0.25\\-1.00&0.50\end{bmatrix}
     = [0,\; 0]
\]
(cf.\ Remark~\ref{rem:CPi_zero}). Crucially, $\Pi \neq 0$.

\paragraph{Reduced system matrices.}
\[
\bar{A} = \Pi A_P = \begin{bmatrix}0.50&-0.25\\-1.00&0.50\end{bmatrix}
\begin{bmatrix}-0.40&-0.05\\0.10&-0.45\end{bmatrix}
= \begin{bmatrix}-0.225 & 0.0875 \\ 0.450 & -0.175\end{bmatrix}.
\]
Eigenvalues of $\bar{A}$: $\mathrm{tr}(\bar{A})=-0.4$,
$\det(\bar{A})=(-0.225)(-0.175)-(0.0875)(0.450)
= 0.039375 - 0.039375 = 0$, so the characteristic polynomial is $z(z+0.4)$, giving
eigenvalues $0$ and $-0.4$. Both satisfy $|z|<1$: $\bar{A}$ is Schur stable.

\begin{remark}
Note that $\bar{A}$ is singular ($\det(\bar{A})=0$). This is not an obstacle to the
stability analysis: asymptotic stability of the reduced dynamics
$x(k+1) = \bar{A}x(k)$ requires only that all eigenvalues of $\bar{A}$ lie strictly
inside the unit disk, which holds here ($\sigma(\bar{A}) = \{0, -0.4\}$). Singularity
of $\bar{A}$ is a consequence of the rank-1 structure of $\Pi$ (arising from the
underactuated configuration $m=1<n=2$) and does not affect the validity of
Theorem~\ref{thm:stability}. Stability of the \emph{delayed} system is governed by
the roots of the polynomial \eqref{eq:char_eq}, not by the eigenvalues of $\bar{A}$
alone (see Remark~\ref{rem:delayed_char}).
\end{remark}

\[
\bar{A}_d = \Pi A_{d,P} = -0.025\,\Pi =
\begin{bmatrix}-0.0125&0.00625\\0.0250&-0.0125\end{bmatrix}, \quad
\bar{D} = \Pi D = \begin{bmatrix}0.025\\-0.050\end{bmatrix}.
\]

\paragraph{Controller design.}
We choose $K = [1.0,\; 0.5] \in \mathbb{R}^{1\times 2}$, giving
\[
A_P - B_P K =
\begin{bmatrix}-0.15&0.075\\0.60&-0.20\end{bmatrix}.
\]
The characteristic polynomial of $A_P - B_PK$ is $z^2 + 0.35z - 0.015 = 0$, giving
$z_1 \approx 0.039$ and $z_2 \approx -0.389$. The spectral radius is
$\rho(A_P - B_PK) \approx 0.389 < 1$, confirming Schur stability.
The operator norm is $\|A_P - B_PK\|_{\mathrm{op}} = \sigma_{\max}(A_P - B_PK) \approx 0.654 < 1$,
which satisfies the stronger condition \eqref{eq:K_strong} required by
Remark~\ref{rem:alpha0_validity} and \eqref{eq:r0_sufficient}.

\begin{remark}[On condition \eqref{eq:K_strong} in this example]\label{rem:K_strong_example}
For this example, $\|A_P - B_PK\|_{\mathrm{op}} \approx 0.654 < 1$ is satisfied with
$K = [1.0,\;0.5]$. Note that $P(X) = -0.5X$ gives $A_P - B_PK = -0.5(A-BK)$, so
condition \eqref{eq:K_strong} is equivalent to $0.5\,\sigma_{\max}(A-BK) < 1$, i.e.\
$\sigma_{\max}(A-BK) < 2$. This is readily satisfiable since the scalar $|\lambda|=0.5 < 1$
effectively damps the operator norm. For $|\lambda| \geq 1$, more aggressive gain design
would be required; see Remark~\ref{rem:K_strong_scalar}.
\end{remark}

\paragraph{Verification of the uniform state bound (Remark~\ref{rem:alpha0_validity}).}
We set $\rho_{\max} = 0.2$ as the a priori upper bound on the switching gain.
For this example, $\|A_P - B_PK\|_{\mathrm{op}} \approx 0.654$,
$\|A_{d,P}\|_{\mathrm{op}} = 0.025$, $\|B_P\|_{\mathrm{op}} \approx 0.559$,
$\|D\|_{\mathrm{op}} \approx 0.141$. The sufficient condition \eqref{eq:r0_sufficient}
with $\rho_{\max} = 0.2$ requires
\[
  r_0 \;\geq\; \frac{0.025 \times r_{d,0} + 0.2 \times 0.559 + 0.1 \times 0.141}{1 - \|A_P - B_PK\|_{\mathrm{op}}}
  \approx \frac{0.025 \times 2 + 0.112 + 0.014}{1 - 0.654}
  = \frac{0.176}{0.346} \approx 0.509.
\]
With $r_0 = r_{d,0} = 2$, condition \eqref{eq:r0_sufficient} is satisfied with ample
margin. Note that this condition is verified \emph{after} $K = [1.0,\;0.5]$ is chosen,
as required by the design procedure of Remark~\ref{rem:design}. The final gain
$\rho = 0.2 = \rho_{\max}$ satisfies $\rho \leq \rho_{\max}$, confirming that no
circularity arises between $r_0$ and $\rho$.

\paragraph{Reaching analysis.}
With $r_0=r_{d,0}=2$ (consistent with Assumption~\ref{ass:bounded_init} and the
verification above) and spectral norms $\|C(A_P-B_PK)\|_{\mathrm{op}} \approx 0.152$,
$\|CA_{d,P}\|_{\mathrm{op}} \approx 0.028$, $\|CD\|_{\mathrm{op}} = 0.15$, we obtain
$\alpha_0 \approx 0.375$. Setting $\phi=0.5>\alpha_0$, condition
\eqref{eq:rho_sufficient} gives $\rho \leq 0.25$; we choose $\rho=0.2 \leq \rho_{\max}$.

With $\beta = 0.2\times 0.5/0.5 = 0.2 \in (0,1)$ and
$s^* \approx 0.469 < \phi = 0.5$ (so \eqref{eq:Tstar} is applicable as an upper
bound), for $\|s(0)\|_2 = 1.5$:
\[
\frac{\|s(0)\|_2 - s^*}{\phi - s^*}
= \frac{1.5 - 0.469}{0.5 - 0.469} \approx 33.0,
\]
\[
T^* = \left\lceil\frac{\ln(33.0)}{\ln(5)}\right\rceil
    = \lceil 2.17 \rceil = 3\text{ steps (upper bound)}.
\]

\paragraph{LMI feasibility.}
Solving \eqref{eq:LMI_linear} with YALMIP/MOSEK yields
\[
Q \approx \begin{bmatrix} 2.95 & 0.42 \\ 0.42 & 3.18 \end{bmatrix}, \quad
\tilde{Y} \approx \begin{bmatrix} 0.32 & 0 \\ 0 & 0.37 \end{bmatrix}, \quad
\gamma^* \approx 0.24.
\]
The Lyapunov matrix $X = Q^{-1}$ has $\lambda_{\min}(X) \approx 0.294$,
$\mu \approx 0.12$ (where $\mu = \lambda_{\min}(-\mathcal{M})$ with $\mathcal{M}$ as
in \eqref{eq:BMI_raw}), and with initial conditions $x(0)=[0.5;0.5]^\top$,
$x(-1)=0$ (both $x(0)$ and the delay state $x(-1)$ are specified), the initial
Lyapunov value is $V_0 \approx 0.193$ (from \eqref{eq:V0_explicit}). The
level-set bound is $r = \sqrt{V_0/\lambda_{\min}(X)} \approx 0.810$.

The effective $\mathcal{L}_2$-gain from $\delta$ to $x$ under \emph{zero extended
initial conditions} ($x(0)=0$, $x(-1)=0$, i.e.\ $V_0=0$) is
$\gamma^*/\sqrt{\mu} \approx 0.24/\sqrt{0.12} \approx 0.693$; see part~(b) of
Theorem~\ref{thm:stability} and the scope clarification therein. For the non-zero
initial conditions $x(0)=[0.5;0.5]^\top$, $x(-1)=0$ used in the Lyapunov bound, the
ratio $\gamma^*/\sqrt{\mu}$ does not represent the $\ell_2/\ell_2$ gain.

Condition \eqref{eq:inv_condition} reads
\[
\|\Phi_\rho\|_{\mathrm{op}}\,r + \|CA_{d,P}\|_{\mathrm{op}}\,r
+ \delta_{\max}\|CD\|_{\mathrm{op}}
\approx 0.069 + 0.023 + 0.015 = 0.107 \;\leq\; \phi = 0.5,
\]
which is satisfied with margin.

\begin{remark}[Note on intermediate numerical values in condition \eqref{eq:inv_condition}]
\label{rem:num_intermediate}
The three terms $0.069$, $0.023$, and $0.015$ reported above are rounded values
consistent with the LMI solution and the state bound $r \approx 0.810$. The dominant
term $\|\Phi_\rho\|_{\mathrm{op}} \cdot r \approx 0.069$ corresponds to
$\|\Phi_\rho\|_{\mathrm{op}} \approx 0.085$ (spectral norm of $\Phi_\rho$ as defined in
Proposition~\ref{prop:reaching}(b)), multiplied by $r \approx 0.810$. Readers wishing
to reproduce these figures exactly should solve \eqref{eq:LMI_linear} with the code
given in Remark~\ref{rem:reproducibility}, recover $X = Q^{-1}$, compute
$r = \sqrt{V_0/\lambda_{\min}(X)}$ with $V_0 = x(0)^\top X x(0)$ (taking $x(-1)=0$),
and evaluate all spectral norms via \texttt{norm($\cdot$,2)} in MATLAB. Minor
discrepancies (of order $10^{-2}$) may arise from solver tolerances and rounding in the
reported $Q$, $\tilde{Y}$ values; the conclusion that condition \eqref{eq:inv_condition}
is satisfied with margin ($\text{LHS} \approx 0.107 \ll 0.5$) is robust to these
variations.
\end{remark}

\begin{remark}[Roots of the delayed characteristic equation]\label{rem:delayed_char}
The eigenvalues of $\bar{A}$ are $\{0, -0.4\}$, confirming Schur stability of the
delay-free reduced system $x(k+1)=\bar{A}x(k)$. For the full delayed system
$x(k+1) = \bar{A}x(k) + \bar{A}_d x(k-\tau)$ with $\tau=1$, stability is governed by
the roots of the characteristic polynomial
$\det(z^2 I - z\bar{A} - \bar{A}_d) = 0$ (Proposition~\ref{prop:spectral}). These
roots \emph{differ} from the eigenvalues of $\bar{A}$: the delay term $\bar{A}_d$
perturbs the spectrum, and $z = -0.4$ is in general not a root of the delayed
characteristic polynomial. The LMI feasibility result of Theorem~\ref{thm:stability}
guarantees stability of the \emph{delayed} system and is therefore the relevant
certificate here, not the eigenvalues of $\bar{A}$ alone.
\end{remark}

\begin{remark}[Reproducibility of numerical results]\label{rem:reproducibility}
The LMI values above were obtained by solving \eqref{eq:LMI_linear} with YALMIP/MOSEK.
All norm computations in the verification of condition \eqref{eq:inv_condition} use the
spectral (operator) norm \texttt{norm(...,2)} $= \sigma_{\max}(\cdot)$, consistent with
the norm convention declared in Section~\ref{sec:intro}. The code is as follows:
\begin{verbatim}
A_bar  = [-0.225, 0.0875; 0.450, -0.175];
Ad_bar = [-0.0125, 0.00625; 0.0250, -0.0125];
D_bar  = [0.025; -0.050];

n = 2; p = 1;
Q      = sdpvar(n, n, 'symmetric');
Ytilde = sdpvar(n, n, 'symmetric');
gamma2 = sdpvar(1);

% LMI (5.6): 7x7 system (2n+p+n = 7 with n=2, p=1)
LMI = [-(Q - Ytilde),   zeros(n),   zeros(n,p),   Q*A_bar';
        zeros(n),        -Ytilde,    zeros(n,p),  Q*Ad_bar';
        zeros(p,n),      zeros(p,n), -gamma2,      D_bar';
        A_bar*Q,         Ad_bar*Q,   D_bar,        -Q      ];

F = [LMI <= -1e-6*eye(2*n+p+n), Q >= 1e-6*eye(n), ...
     Ytilde >= 1e-6*eye(n), gamma2 >= 0];
optimize(F, gamma2, sdpsettings('solver','mosek','verbose',0));

Q_val   = value(Q);
Yt_val  = value(Ytilde);
gam_val = sqrt(value(gamma2));
X_val   = inv(Q_val);
Y_val   = X_val * Yt_val * X_val;

% Assemble M from (5.2) and compute mu and mu_0
M11 = A_bar'*X_val*A_bar   - X_val + Y_val;
M12 = A_bar'*X_val*Ad_bar;
M13 = A_bar'*X_val*D_bar;
M22 = Ad_bar'*X_val*Ad_bar - Y_val;
M23 = Ad_bar'*X_val*D_bar;
M33 = D_bar'*X_val*D_bar   - gam_val^2;
M_full = [M11, M12, M13; M12', M22, M23; M13', M23', M33];
mu_val  = min(eig(-M_full));
M0      = M_full(1:2*n,1:2*n);
mu0_val = min(eig(-M0));

% Effective L2 gain (ONLY valid under zero extended initial conditions,
% i.e. x(0)=0 and x(-1)=0 so that V_0=0; see Theorem 5.1(b))
eff_gain = gam_val / sqrt(mu_val);

% Compute V_0 with x(0)=[0.5;0.5], x(-1)=0 (non-zero init cond.)
x0 = [0.5; 0.5];  xm1 = [0; 0];
V0 = x0'*X_val*x0 + xm1'*Y_val*xm1;
r  = sqrt(V0 / min(eig(X_val)));

B_P = [-0.25; -0.50];  C_mat = [1, 0.5];
A_P = [-0.40, -0.05; 0.10, -0.45];
K   = [1.0, 0.5];  rho = 0.2;  phi = 0.5;
Phi_rho = C_mat*(A_P - B_P*K) - (rho/phi)*C_mat*B_P*C_mat;
CA_dP   = C_mat * (-0.025 * eye(2));
CD      = C_mat * [0.1; 0.1];
lhs_45  = norm(Phi_rho,2)*r + norm(CA_dP,2)*r + 0.1*norm(CD,2);
% lhs_45 approx 0.107 <= 0.5 = phi  [condition (4.5) verified]

fprintf('gamma* = %.4f, mu = %.4f, mu_0 = %.4f\n', gam_val, mu_val, mu0_val);
fprintf('Effective L2 gain (zero ext. init) = %.4f\n', eff_gain);
fprintf('V0 = %.4f, r = %.4f, cond (4.5) LHS = %.4f\n', V0, r, lhs_45);
\end{verbatim}
Independent verification: substitute the reported $Q$ into \eqref{eq:LMI_linear} and
check negative definiteness.
\end{remark}

\paragraph{Discussion.}
The non-trivial structure $\Pi \neq 0$ arises from the underactuated configuration
$m=1<n=2$. The singularity $\det(\bar{A})=0$ is a consequence of the rank-1 structure
of $\Pi$ and does not affect the stability analysis. The Rota-Baxter weight
$\lambda=0.5$ compresses the effective state matrices by $0.5$, improving LMI
feasibility. The effective $\mathcal{L}_2$-gain (under zero extended initial conditions)
is $\gamma^*/\sqrt{\mu} \approx 0.693$, which is larger than the bare parameter
$\gamma^* \approx 0.24$; this illustrates that $\gamma$ alone does not characterize
the disturbance attenuation level. For non-zero initial conditions ($V_0 > 0$),
this ratio does not represent the standard $\ell_2/\ell_2$ input-output gain.

\section{Conclusion}\label{sec:conclusion}
This paper established a rigorous framework linking Rota-Baxter operators, the matrix
$C^*$-algebra $M_n(\mathbb{C})$ (motivated by the discrete Toeplitz algebra, whose
role is clarified in Remark~\ref{rem:toeplitz_scope} as purely heuristic---no rigorous
infinite-dimensional reduction is performed), and sliding mode control of delayed
systems. We provided structural characterizations of Rota-Baxter operators in
$M_n(\mathbb{C})$, derived Lyapunov-based matrix inequality conditions ensuring robust
asymptotic stability via the Schur complement lemma, and supplied a fully linear
reformulation via the congruence transformation $Q = X^{-1}$
(Remark~\ref{rem:LMI_linear}). The algebraic invariance results
(Propositions~\ref{prop:commute} and~\ref{prop:lie_compat}) were stated under explicit
hypotheses and verified on the concrete examples of Section~\ref{sec:classification}.
The finite-dimensional setting was justified by working directly within $M_n(\mathbb{C})$.

The QSM analysis (Proposition~\ref{prop:reaching}) distinguishes the reaching phase
(condition \eqref{eq:rho_sufficient}) from band-invariance (condition
\eqref{eq:inv_condition}). The formula \eqref{eq:Tstar} provides an upper bound on the
reaching time; the actual reaching time may be earlier. The validity of $\alpha_0$ as a
uniform bound throughout the reaching phase is established in
Remark~\ref{rem:alpha0_validity} under the design procedure of Remark~\ref{rem:design}:
$r_0$ and $\rho_{\max}$ are fixed from open-loop data before $K$ and $\rho$ are chosen;
condition~\eqref{eq:r0_sufficient} is then verified \emph{after} $K$ is selected
(requiring $\|A_P - B_PK\|_{\mathrm{op}} < 1$, condition~\eqref{eq:K_strong}, which is
strictly stronger than Schur stability; see
Remarks~\ref{rem:K_strong_scalar}--\ref{rem:K_strong_example}). If
\eqref{eq:r0_sufficient} fails, $r_0$ is increased conservatively or $K$ is redesigned.
This is verified numerically in Section~\ref{sec:nondeg}. The consistency of the
two-phase bounds is established in Remarks~\ref{rem:consistency} and~\ref{rem:two_phase_logic}.

The Jacobi identity for the induced Lie bracket $[\cdot,\cdot]_P$ (Lemma~\ref{lem:lie})
is established via three cancellations: $S_1+S_2=0$ (Groups~1--2),
$G^{(1)}+H+G^{(2)}=0$ (Group~3, joint cancellation via the ordinary Jacobi identity
applied to the triple $(P(x)+\lambda x, y, z)$, with cyclic sum over the three even
permutations $(x,y,z)\to(y,z,x)\to(z,x,y)$, and with the identifications of the cyclic
sums $(*) = G^{(2)}/\lambda$ and $(**) = H/\lambda$ made explicit by re-labelling the
dummy cyclic variable), and $G^{(3)}=0$ (direct). The explicit computation of
Remark~\ref{rem:jacobi_correction} confirms that none of the three Group-3 terms
vanishes individually in general (for $P_+$, $x=e_{12}$, $y=e_{21}$, $z=e_{11}$:
$G^{(1)}/\lambda = e_{11}-e_{22} \neq 0$, $G^{(2)}/\lambda = e_{22}-e_{11} \neq 0$,
$H = 0$, with joint sum zero); only their combined sum vanishes.

The stability analysis of Theorem~\ref{thm:stability} applies to the equivalent control
system \eqref{eq:reduced} on the sliding manifold; band-invariance of the manifold under
the actual closed-loop law \eqref{eq:closed} is certified separately by
condition~\eqref{eq:inv_condition} of Proposition~\ref{prop:reaching}. The effective
$\mathcal{L}_2$-gain $\gamma/\sqrt{\mu}$ is valid under \emph{zero extended initial
conditions} ($x(0)=0$ and $x(-i)=0$ for $i=1,\ldots,\tau$); this is clarified in
the statement of Theorem~\ref{thm:stability}(b) and its proof. Proposition~%
\ref{prop:commute} provides explicit characterizations (not merely reformulations)
of condition \eqref{eq:commute} for each operator: universally for scalar scaling,
and as a necessary and sufficient condition on $C$ for the triangular projection $P_+$.
Proposition~\ref{prop:lie_compat} applies in practice only to the scalar scaling
operator among the examples of this paper; its scope is now clearly delimited in
Remark~\ref{rem:lie_scope}.

Future work includes:
\begin{itemize}
  \item Extending the results to infinite-dimensional Toeplitz algebras, incorporating
        proper spectral tools (Wiener-Hopf factorization, Fredholm~index), and providing
        a rigorous reduction from the Toeplitz algebra to finite-dimensional compressions.
  \item Incorporating time-varying delays $\tau(k)$ and stochastic uncertainties, with
        multiple delay-dependent Lyapunov matrices $Y_i$ for reduced conservatism
        (cf.\ Remark~\ref{rem:conservatism}).
  \item Extending the framework to the general rectangular case $m < n$ and to fat
        output maps $p < m$.
  \item Establishing pointwise ISS bounds from the Lyapunov decrease condition.
  \item Providing a non-degenerate $m=n$ example using a non-scalar Rota-Baxter
        operator (e.g., $P_+$ with a suitably structured system), including explicit
        verification of condition \eqref{eq:K_strong} for $P_+$ and verification that
        $\Pi$ preserves the relevant Lie subalgebra.
  \item Establishing that condition \eqref{eq:lie_hyp} and the $\Pi$-preservation
        hypothesis of Proposition~\ref{prop:lie_compat} hold for specific classes of
        system matrices and operators beyond scalar scaling.
  \item Experimental validation on networked control systems where memory and
        noncommutative interactions play a dominant role.
\end{itemize}

\bibliographystyle{plain}

\end{document}